\documentclass[a4paper]{article}
\usepackage{amsmath,amsfonts,amssymb,amsthm} 
\usepackage[all]{xy}
\usepackage[english]{babel}

\numberwithin{equation}{section}

\usepackage{vmargin}
\setmarginsrb{18mm}{15mm}{18mm}{20mm}
          {10mm}{7mm}{10mm}{10mm}

\theoremstyle{plain}
\newtheorem{teo}{Theorem}[section]
\newtheorem{prop}[teo]{Proposition}
\newtheorem{pro}[teo]{Problem}
\newtheorem{lemma}[teo]{Lemma}
\newtheorem{cor}[teo]{Corollary}
\newtheorem{fact}[teo]{Fact}
\newtheorem{que}[teo]{Question}

\theoremstyle{definition}

\newtheorem{defin}[teo]{Definition}

\theoremstyle{remark}

\newtheorem{rem}[teo]{Remark}
\newtheorem{exa}[teo]{Example}

\newtheorem{claim}[teo]{Claim}
\newtheorem{remark}[teo]{Remark}

\DeclareMathOperator{\CH}{\mathfrak{Char}}
\newcommand{\sv}{s_{\vs}}
\newcommand{\su}{s_{\us}}
\newcommand{\n}{n}

\newcommand{\G}{G_\delta}

\newcommand{\FF}{F_{\sigma\delta}}

\newcommand{\T}{\mathbb{T}}
\newcommand{\Z}{\mathbb{Z}}
\newcommand{\Q}{\mathbb{Q}}
\newcommand{\R}{\mathbb{R}}
\newcommand{\Jp}{\mathbb{J}_p}
\newcommand{\N}{\mathbb{N}}

\newcommand{\cc}{\mathfrak{c}}
\newcommand{\Prm}{\mathbb P}

\newcommand{\us}{\mathbf{u}}
\newcommand{\vs}{\mathbf{v}}
\newcommand{\ws}{\mathbf{w}}
\newcommand{\la}{\left|}
\newcommand{\ra}{\right|}
\newcommand{\lb}{\left\|}
\newcommand{\rb}{\right\|}
\newcommand{\lgr}{\left\{}
\newcommand{\rgr}{\right\}}
\newcommand{\wh}{\widehat}
\newcommand{\wt}{\widetilde}

\newcommand{\rs}{\restriction}

\DeclareMathOperator{\g}{\mathfrak{g}}

\newcommand{\Gz}{\Gamma_\vs^0}
\newcommand{\Gi}{\Gamma_\vs^\infty}

\author{D. Dikranjan\footnote{The first and second named authors were partially supported by Fondazione Cassa di Risparmio di Padova e Rovigo (Progetto di Eccellenza ``Algebraic structures and their applications'')}\\ \footnotesize{dikran.dikranjan@uniud.it} \and A. Giordano Bruno$^*$\footnote{The second named author is supported by Programma SIR 2014 by MIUR (Project GADYGR)}\\ \footnotesize{anna.giordanobruno@uniud.it} \and D. Impieri\\ \footnotesize{daniele.impieri@uniud.it}}

\date{\footnotesize{Universit\`a degli Studi di Udine \\ Dipartimento di Matematica e Informatica \\ Via delle Scienze 206, Udine - Italy}}

\title{Characterized subgroups of topological abelian groups}

\begin{document}

\maketitle

\abstract{A subgroup $H$ of a topological abelian group $X$ is said to be characterized by a sequence $\vs =(v_n)$ of characters of $X$ if $H=\{x\in X:v_n(x)\to 0\ \text{in}\ \T\}$. We study the basic properties of characterized subgroups in the general setting, extending results known in the compact case. For a better description, we isolate various types of characterized subgroups. Moreover, we introduce the relevant class of autochacaracterized groups (namely, the groups that are characterized subgroups of themselves by means of a sequence of non-null characters); in the case of locally compact abelian groups, these are proved to be exactly the non-compact ones. As a by-product of our results, we find a complete description of the characterized subgroups of discrete abelian groups.}

\bigskip 



\section{Introduction}

For a topological abelian group $X$, we denote by $\widehat X$ its dual group, that is, the group of all characters of $X$ (i.e., continuous homomorphisms $X\to \T$). Following \cite{DMT05}, for a sequence of characters $\vs=(v_n)\in\widehat{X}^\N$, let 
$$
\sv(X):=\lgr x\in X: v_n(x)\to0\rgr,
$$ which is always a subgroup of $X$. A subgroup $H$ of $X$ is said to be \emph{characterized} if $H=\sv(X)$ for some $\vs=(v_n)\in\widehat X^\N$. 

Historically, characterized subgroups were studied exclusively in the case of the circle group $\T=\R/\Z$ (see \cite{Arm81,BDS01,DPS90,Lar88}), also in relation with Diophantine approximation, Dynamical systems and Ergodic theory (see \cite{BDS01,PS73,Win02}). (One can find more on this topic in the nice survey \cite{Gab12}, as well as in the more recent  \cite{DI15a, DI15c,Gab14,Neg14}.) 
Some general results were then obtained in the case of metrizable compact abelian groups; for example, it is known that every countable subgroup of a metrizable compact abelian group is characterized (see \cite[Theorem 1.4]{DK07} and \cite{BSW06}), and it was pointed out in \cite{BSW06,Dik06} that the metrizability is necessary, as a compact abelian group with a countable characterized subgroup is necessarily metrizable. Only recently, the case of general compact abelian groups was given full attention in \cite{DG13} and a reduction theorem (to the metrizable case) was obtained. 

The few exceptions \cite{Bor83,Gab12} (concerning respectively characterized subgroups of $\R$ and of compact non-abelian groups), only confirm the tendency to study the characterized subgroups of $\T$ or, more recently, of compact abelian groups. 
To say the least, even the simplest case of characterized subgroups of \emph{discrete} abelian groups has never been considered in the literature to the best of our knowledge. 

\medskip		
The aim of these notes is to develop a general approach to characterized subgroups of arbitrary topological abelian groups, collecting the basic properties so far established in the compact case.

\bigskip
We isolate three special types of characterized subgroups, namely $T$-characterized, $K$-characterized and $N$-characterized subgroups (see Definition \ref{T-N-K}). 
Of those, $T$-characterized subgroups were introduced by Gabriyelyan in \cite{Gab14}, $K$-characterized subgroups were substantially studied by Kunen and his coauthors in \cite{HK05,HK06}, while $N$-characterized subgroups, even if never introduced explicitly, have been frequently used in the theory of duality in topological abelian groups (being nothing else but the annihilators of countable sets of the dual group). One of the advantages of this articulation is the possibility to establish some general permanence properties that fail to be true in the whole class of characterized subgroups, but hold true in some of these subclasses. Moreover, we see that each characterized subgroup is the intersection of an $N$- and a $K$-characterized subgroup (see Corollary \ref{NcapK}).

Inspired by the notion of $T$-characterized subgroup, we introduce also the stronger one of $TB$-characterized subgroup (see Definition \ref{TB}).  The following implications hold and none of them can be reversed in general (see Section \ref{hsec}):
\begin{equation*}
\xymatrix{
\text{$TB$-characterized} \ar@{=>}[r] & \text{$T$-characterized}\ar@{=>}[r] & \text{$K$-characterized}\ar@{=>}[r] & \text{characterized}\\
\text{proper dense characterized} \ar@{=>}[u]  & & \text{$N$-characterized}\ar[u]_{(*)} \ar@{=>}[r]  & \text{closed characterized}\ar@{=>}[u]
}
\end{equation*}
where ($*$) holds under the assumption that the subgroup is closed and has infinite index (see Corollary \ref{N->K}).

\smallskip
In Section \ref{autosec} we introduce the prominent class of autocharacterized groups (see Definition \ref{PA}). These are the topological abelian groups that are characterized subgroups of themselves by means of a \emph{non-trivial sequence} of characters (see \eqref{ev0}). 
The  fact that compact abelian groups are not autocharacterized is equivalent to the well known non-trivial fact that the Bohr topology of an infinite discrete abelian group has no non-trivial convergent sequences. Here we generalize this fact by proving that the property of being non-autocharacterized describes the compact abelian groups within the class of all locally compact abelian groups (see Theorem \ref{compactnonac}). Moreover, in the general case we describe the autocharacterized groups in terms of their Bohr compactification (see Theorem \ref{autocarcar}).

We study the basic properties of $K$- and of $N$-characterized subgroups respectively in Sections \ref{Ksec} and \ref{Nsec}. For the case of discrete abelian groups, which is considered here for the first time, we give a complete description of characterized subgroups by showing that these are precisely the subgroups of index at most $\cc$, and that a subgroup is characterized precisely when it is $K$- and $N$-characterized (see Corollary \ref{DiscGr}).

In Section \ref{Ksec}, we describe when a closed subgroup of infinite index is both $K$- and $N$-characterized, and we see that this occurs precisely when it is only $N$-characterized (see Theorem \ref{T1122}); then we consider the special case of open subgroups, proving that proper open subgroups of infinite index (respectively, of finite index) are $K$-characterized if and only if they are characterized (respectively, autocharacterized) (see Theorems \ref{Kcar=car} and \ref{PtimesKchar:new}).  In particular, no proper open subgroup of a compact abelian group is $K$-characterized.

In Section \ref{Nsec}, extending a criterion for compact abelian groups given in \cite{DG13}, we show that for locally compact abelian groups one can reduce the study of characterized subgroups to the metrizable case (see Theorem \ref{ThB}). Moreover, we describe the closed characterized subgroups of the locally compact abelian groups by showing that they are precisely the $N$-characterized subgroups (see Theorem \ref{Titemi}). As a consequence, we add other equivalent conditions to the known fact from \cite{DG13} that a closed subgroup of a compact abelian group is characterized if and only if it is $\G$, namely, that the subgroup is $K$- and $N$-characterized (see Theorem \ref{CKNG}).

Section \ref{SSMM} concerns $T$-characterized subgroups of compact abelian groups. We establish a criterion to determine when a characterized subgroup of a compact abelian group is not $T$-characterized (see Theorem \ref{MainThChapter7}), 
which extends results from \cite{Gab14}. The impact on characterized subgroups of connected compact abelian groups is discussed.

The final Section \ref{finalsec} contains various comments and open problems, both general and specific.

\subsection*{Acknowledgements}

It is a pleasure to thank Wis Comfort for his useful comments related to (an equivalent form of) Question \ref{Wis}, and Saak Gabriyelyan for sending us his preprint \cite{Gab14} in 2012.

\subsection*{Notation and terminology}

The symbols $\Z$, $\Prm$, $\N$ and  $\N_+$ are used for the set of integers, the set of primes, the set of non-negative integers and the set of positive integers, respectively. The circle group $\T$ is identified with the quotient group $\R/\Z$ of the reals $\R$ and carries its usual compact topology.  We denote by $\T_+$ the image of $[-1/4,1/4]$ in $\T$. If $m$ is a positive integer, $G[m]=\{x\in G:m x=0\}$ and $\Z(m)$ is the cyclic group of order $m$. 
Moreover, for $p\in\Prm$, we denote by $\Z(p^\infty)$ and $\Jp$ respectively the Pr\"ufer group and the $p$-adic integers.

We say that an abelian group $G$ is \emph{torsion} if every element of $G$ is torsion (i.e., for every $x\in G$ there exists $m\in\N_+$ such that $mx=0$).
If $M$ is a subset of $G$, then $\langle M\rangle$ is the smallest subgroup of $G$ containing $M$. 

For a topological space $X=(X,\tau)$ the \emph{weight} $w(X)$ of $X$ is the minimum cardinality of a base for $\tau$. For a subset $A$ of $X$ we denote by $\overline{A}^\tau$ the closure of $A$ in $(X,\tau)$ (sometimes we write only $\overline A$ when there is no possibility of confusion).

A topological abelian group $X$ is \emph{totally bounded} if for every open subset $U$ of $0$ in $X$ there exists a finite subset $F$ of $X$ such that $U+F=X$. If $X$ is totally bounded and Hausdorff we say that $X$ is \emph{precompact}.
We denote by $\widetilde X$ the two-sided completion of $X$; in case $X$ is precompact $\widetilde X$ coincides with the Weil completion.

For a subset $A$ of $X$, the \emph{annihilator} of $A$ in $\wh X$ is $A^\perp=\{\chi\in\wh X:\chi(A)=\{0\}\}$, and for a subset $B$ of $\wh X$, the \emph{annihilator} of $B$ in $X$ is $B^\perp=\{x\in X:\chi(x)=0\ \text{for every}\ \chi\in B\}$.

We say that a sequence $\vs\in\wh X^\N$ is \emph{non-trivial} if it is not eventually null.

\section{Background on topological groups}

\subsection{Basic definitions}

Let $G$ be an abelian group and $H$ a subgroup of $\mathrm{Hom(G,\T)}$. Let $T_H$ be the weakest group topology on $G$ such that all characters of $H$ are continuous with respect to $T_H$; then $T_H$ is totally bounded. Viceversa, Comfort and Ross proved that any totally bounded group topology is of this type (see \cite[Theorem 1.2]{CR64}).

\begin{teo}\label{def}\emph{\cite[Theorems 1.2, 1.3 and 1.11, Corollary 1.4]{CR64}}
Let $G$ be an abelian group and $H$ a subgroup of $G$. Then $T_H$ is totally bounded and:
\begin{itemize}
\item[(a)] $T_H$ is Hausdorff if and only if $H$ separates the points of $G$;
\item[(b)]  $T_H$ is metrizable if and only if $H$ is countable. 
\end{itemize}
\end{teo}

The following two notions will be often used in the paper.

\begin{defin}
A topological abelian group $X$ is said to be:
\begin{itemize}
\item[(i)] \emph{maximally almost periodic} (briefly, \emph{MAP}) if $\wh X$ separates the points of $X$; 
\item[(ii)] \emph{minimally almost periodic} (briefly, \emph{MinAP}) if $\wh X=\{0\}$. 
\end{itemize}
\end{defin}

We recall that two group topologies $\tau_1$ and $\tau_2$ on an abelian group $X$ are \emph{compatible} if they have the same characters, that is, $\widehat{(X,\tau_1)}=\widehat{(X,\tau_2)}$.

If $X=(X,\tau)$ is a topological abelian group, denote by $\tau^+$ its \emph{Bohr topology}, that is, the finest totally bounded group topology on $X$ coarser than $\tau$ (indeed, $\tau^+=T_{\wh X}$); we denote $X$ endowed with its Bohr topology also by $X^+$ and we call $\tau^+$ also the \emph{Bohr modification} of $\tau$. Clearly, $\tau$ and $\tau^+$ are compatible. Moreover, 
\begin{itemize}
\item[(i)] $\tau$ is MAP if and only if $\tau^+$ is Hausdorff; 
\item[(ii)] $\tau$ is MinAP if and only if $\tau^+$ is indiscrete. 
\end{itemize}

A subgroup $H$ of $(X,\tau)$ is: 
\begin{itemize}
\item[(a)] \emph{dually closed} if $H$ is $\tau^+$-closed (or, equivalently, $X/H$ is MAP); 
\item[(b)] \emph{dually embedded} if every $\chi \in\wh H$ can be extended to $X$. 
\end{itemize}

Clearly, dually closed implies closed, since $H\leq\overline H^\tau\leq \overline H^{\tau^+}$. 

\begin{fact}\label{lcafact}
Let $X$ be a locally compact abelian group. Then:
\begin{itemize}
\item[(i)] every closed subgroup $H$ of $X$ is dually closed, i.e., $X/H$ is MAP;
\item[(ii)] in particular, every locally compact abelian group is MAP, and 
\item[(iii)] $X$ and $X^+$ have the same closed subgroups;
\item[(iv)] consequently, $X$ is separable if and only if $X^+$ is separable.
\end{itemize}
\end{fact}

For a topological abelian group $X$ and a subgroup $L$ of $X$, the \emph{weak topology} $\sigma(\widehat X, L)$ of the dual $\widehat X$ is the totally bounded group topology of $\widehat X$ generated by the elements of $L$ considered as characters of $\widehat X$; namely, for every $x\in L$, consider $\xi_x:\wh{X}\to\T$ defined by $\xi_x(\chi)=\chi(x)$ for every $\chi\in\wh{X}$. A local base of $\sigma(\wh{X},L)$ is given by the finite intersections of the sets $\xi_x^{-1}(U)$, where $x\in L$ and $U$ is an open neighborhood of $0$ in $\T$. 
Clearly, if $L_1\leq L_2$, then $\sigma(\wh X,L_1)\leq\sigma(\wh X,L_2)$. 

Note that the weak topology $\sigma(\wh X,X)$ is coarser than the compact-open topology on $\wh X$.
If $L$ separates the points of $\wh X$ (e.g, when $L$ is dense in $X$, or when $L=X$), then $\sigma(\wh X,L)$ is precompact.

\begin{fact}\label{sigma=+}
If $X$ is a reflexive topological abelian group, then $\sigma(\widehat X, X)$ coincides with the Bohr topology of $\widehat X$. 
\end{fact}

We recall that a sequence $\vs$ in an abelian group $G$ is a \emph{$T$-sequence} (respectively, \emph{$TB$-sequence}) if there exists a Hausdorff (respectively, precompact) group topology $\tau$ on $G$ such that $\vs$ is a null sequence in $(G,\tau)$.

\begin{lemma}\label{vv}
Let $X$ be a topological abelian group and $\vs\in\wh X^\N$. Then:
\begin{itemize}
\item[(i)] for a subgroup $L$ of $X$, $v_n(x)\to 0$ in $\T$ for every $x\in L$ if and only if $v_n\to 0$ in $\sigma(\wh X,L)$;
\item[(ii)] if $s_\vs(X)$ is dense in $X$, then $\vs$ is a TB-sequence.
\end{itemize}
\end{lemma}
\begin{proof}
(i) follows from the definition of $\sigma(\wh X,L)$. 

(ii) As $s_\vs(X)$ is dense in $X$, then $\sigma(\wh X,s_\vs(X))$ is precompact. By item (a), $v_n\to 0$ in  $\sigma(\wh X,s_\vs(X))$, hence $\vs$ is a $TB$-sequence.
\end{proof}

Let $G$ be a discrete abelian group. For a sequence $\vs\in G^\N$, the group topology 
\begin{equation}\label{sigmaeq}
\sigma_\vs:=T_{s_\vs(\wh G)}
\end{equation}
is the finest totally bounded group topology on $G$ such that $\vs$ is a null sequence in $(G,\sigma_\vs)$.

\begin{fact}\emph{\cite[Lemma 3.1, Proposition 3.2]{DMT05}}
Let $G$ be a discrete abelian group and $\vs\in G^\N$.  The following conditions are equivalent:
\begin{itemize}
\item[(i)] $\vs$ is a $TB$-sequence;
\item[(ii)] $\sigma_\vs$ is Hausdorff;
\item[(ii)] $s_\vs(\wh G)$ is dense in $\wh G$.
\end{itemize} 
\end{fact}

\subsection{Useful folklore results}

We recall the following basic properties that will be used in the paper. Although most of them are well known, we offer proofs for reader's convenience.  

\begin{lemma}\label{perp}
Let $X$ be a topological abelian group and $H$ a subgroup of $X$. Then:
\begin{itemize}
\item[(i)] $\wh{X/H}$ is algebraically isomorphic to $H^\perp$;
\item[(ii)] $\wh X/H^\perp$ is algebraically isomorphic to a subgroup of $\wh H$.
\end{itemize}
\end{lemma}
\begin{proof}
(i) Let $\psi:\wh{X/H}\to \wh X$ be defined by $\chi\mapsto \chi\circ\pi$, where $\pi:X\to X/H$ is the canonical projection. Then $\psi$ is injective and its image is $H^\perp$.

(ii) Let $\rho:\wh X\to \wh H$ be defined by $\chi\mapsto\chi\restriction_H$. Then $\ker\rho=H^\perp$ and so we get the thesis.
\end{proof}

\begin{lemma}\label{HMP}\emph{\cite{DPS90}}
A compact abelian group $K$ is separable if and only if $w(K)\leq\mathfrak c$.
\end{lemma}
\begin{proof} 
The inequality $w(K)\leq\mathfrak c$ holds for every separable regular topological space $K$. 
 
 Assume that $w(K)\leq\mathfrak c$. The discrete abelian group $X=\wh K$ has size $|X|=w(K)\leq\mathfrak c$. Consider the embedding $i:X\to D(X)$, where $D(X)$ is the divisible hull of $X$. Then $|D(X)|\leq\mathfrak c$ and $D(X)=\bigoplus_{i\in I}D_i$, for some countable divisible abelian groups $D_i$ and a set of indices $I$ with $|I|\leq \mathfrak c$. Therefore, $\wh i:\prod_{i\in I}\wh D_i\to \wh X=K$ is a surjective continuous homomorphism and each $\wh D_i$ is a metrizable compact abelian group. By Hewitt-Marczewski-Pondiczery Theorem, since $|I|\leq\mathfrak c$, we have that $\prod_{i\in I}\wh D_i$ is separable, hence $K$ is separable as well.
\end{proof}

\begin{lemma}\label{LinjPre}
Let $X$ be a precompact abelian group. Then $\{0\}$ is $\G$ if and only if there exists a continuous injection $X\to \T^\N$.  
\end{lemma}
\begin{proof} 
If there exists a continuous injection $X\to \T^\N$, then $\{0\}$ is $\G$ in $X$, as it is  $\G$ in $\T^\N$. 

Assume now that $\{0\}=\bigcap_{n\in\N}U_n$, where each $U_n$ is an open subset of $X$, and we can assume that $U_n$ is in the prebase of the neighborhoods of $0$ in $X$.
So for every $n\in\N$ there exist $v_n\in\wh X$ and an open neighborhood $V_n$ of $0$ in $\T$ such that $V_n$ does not contain any non trivial subgroup of $\T$ such that $U_n=v_n^{-1}(V_n)$. Then $\{0\}=\bigcap_{n\in\N}\ker v_n$. Hence, $j:X\to \T^\N$ defined by $j(x)=(v_n(x))_{n\in\N}$ is a continuous injective homomorphism.
\end{proof}

\begin{teo}\label{Tinjection}
Let $X$ be a locally compact abelian group. Then $X$ is metrizable with $|X|\leq \cc$ if and only if there exists a continuous injective homomorphism $X\to\T^\N$.
\end{teo}
\begin{proof} 
If there exists a continuous injective homomorphism $X\to\T^\N$, then clearly $X$ is metrizable and $|X|\leq|\T^\N|=\cc$.

Suppose now that $X$ is metrizable and has cardinality at most $\cc$. It is well known (for example, see \cite{HR79}) that $X=\R^n\times X_0$, where $n\in\N$ and $X_0$ is a locally compact abelian group admitting an open compact (metrizable) subgroup $K$. Clearly, there exist two continuous injective homomorphisms $j_1:\R^n\hookrightarrow\T^\N$ and $j_2:K\hookrightarrow\T^\N$. Therefore, $j_3=(j_1,j_2):\R^n\times K\to\T^\N\times\T^\N\cong\T^\N$ is an injective continuous homomorphism too. Since $\T^\N$ is divisible and $\R^n\times K$ is open in $X$, $j_3$ extends continuously to $\wt{j_3}:X\to \T^\N$. Let $\pi:X\to X/(\R^n\times K)$ be the canonical projection. Since $X/(\R^n\times K)$ is discrete, there exists a continuous injective homomorphism $j_4:X/(\R^n\times K)\to\T^\N$. Let $\varphi=j_4\circ \pi:X\to\T^\N$.
\begin{equation*}
\xymatrix{X\ar[dr]^{\varphi}\ar[d]_{\pi}&\\
X/(\R^n\times K)\ar[r]_{\; \;\;  \;\;\;\; j_4}&\T^\N}
\end{equation*}
Let now $j:X\to \T^\N\times\T^\N\cong\T^\N$ be defined by $j(x)=(\varphi(x),\wt{j_3}(x))$ for every $x\in X$. Then $j$ is continuous since $\varphi$ and $\wt{j_3}$ are continuous. Moreover, $j$ is injective as $j(x)=0$ for some $x\in X$ implies $\varphi(x)=0$ and $\wt{j_3}(x)=0$; therefore, $x\in\R^n\times K$, and so, since $\wt{j_3}\restriction_{\R^n\times K}=j_3$ is injective, one has $x=0$.	
\end{proof}

\section{General permanence properties of characterized subgroups}

Let $X$ be a topological abelian group and denote by $\CH(X)$ the family of all subgroups of $X$ that are characterized.

\medskip
We start by observing that
\begin{equation}\label{ev0}
\text{if}\ \vs\in\wh X^\N\ \ \text{is eventually null, then}\ X=s_\vs(X).
\end{equation}

The following are basic facts on characterized subgroups (see \cite{CTT93,DG13,DK07,DMT05}), we give a proof for reader's convenience.
   	
\begin{lemma}\label{Lbasic}
Let $X$ be a topological abelian group and $\vs\in\widehat{X}^\N$. Then:
\begin{itemize}
\item[(i)] for every subgroup $J$ of $X$, $s_{\vs^*}(J)=s_\vs(X)\cap J$, where $v^*_n=v_n\restriction_{J}$ for every $n\in\N$;
\item[(ii)] $\sv(X)=\su(X)$ if $\us$ is any permutation of $\vs$;
\item[(iii)] $\CH(X)$ is stable under taking finite intersections;
\item[(iv)] $\sv(X)$ is an $\FF$-set (i.e., countable intersection of countable unions of closed subgroups).
\end{itemize}
\end{lemma}
\begin{proof}
Items (i) and (ii) are obvious. To prove (iii), if $\us,\vs\in\wh{X}^\N$, define $\ws=(w_n)$, where $w_{2n}=u_n$ and $w_{2n+1}=v_n$ for every $n\in\N$, hence $\su(X)\cap\sv(X)=s_\ws(X)$. 
To prove (iv), note that 
$\sv(X)=\bigcap_m\bigcup_k\bigcap_{n\ge k} S_{n,m}$, where each $S_{n,m}=\lgr x\in X: \lb v_n(x)\rb\le\frac{1}{m}\rgr$ is a closed subset of $X$. 
\end{proof}

Now we prove that, under suitable hypotheses, the relation of being a characterized subgroup is transitive: 

\begin{prop}\label{NEW:lemma}
Let $X$ be a topological abelian group and $X_0$, $X_1$, $X_2$ subgroups of $X$ with $X_0\leq X_1\leq X_2$ and such that $X_1$ is dually embedded in $X_2$. If $X_0\in \CH(X_1)$ and $X_1\in \CH(X_2)$, then $X_0\in \CH(X_2)$. 
\end{prop}
\begin{proof}
Let $\vs\in\wh{X_1}^\N$ such that $X_0=s_\vs(X_1)$ and let $\ws\in\wh{X_2}^\N$ such that $X_1=s_\ws(X_2)$. 
As $X_1$ is dually embedded in $X_2$, $v_n$ extends to a character $v_n^*$ of $X_2$ for every $n\in\N$, so let $\vs^*=(v_n^*)\in\wh{X_2}^\N$.  Define $\ws^*\in\wh{X_2}^\N$ by letting $w^*_{2n}=v^*_n$ and $w^*_{2n+1}=w_n$ for every $n\in\N$. 
Then, by Lemma \ref{Lbasic}(i), $$s_{\ws^*}(X_2)=s_{\vs^*}(X_2)\cap s_\ws(X_2)=s_{\vs^*}(X_2)\cap X_1=s_\vs(X_1)=X_0,$$ so $X_0\in\CH(X_2)$, as required.
\end{proof}

Clearly, two compatible group topologies have the same characterized subgroups:

\begin{lemma}\label{compatible}
If $\tau_1$ and $\tau_2$ are compatible group topologies on an abelian group $X$, then $\CH(X,\tau_1)=\CH(X,\tau_2)$.
\end{lemma}

In particular, for a topological abelian group $(X,\tau)$, since $\tau$ and its Bohr modification $\tau^+$ are compatible,
$\CH(X,\tau)=\CH(X,\tau^+)$. 


\section{The $\Gamma$-radical}

\begin{defin}
Let $X$ be a topological abelian group. For a subset $\Gamma$ of $\widehat X$,  define the \emph{$\Gamma$-radical} of $X$ by
$$\n_\Gamma(X):=\bigcap_{\chi\in\Gamma}\ker\chi=\Gamma^\perp.$$
\end{defin}

Clearly, $\n_\Gamma(X)$ is a closed subgroup of $X$.

\medskip
The motivation for the choice of the term $\Gamma$-radical is the special case $\Gamma = \wh X$, when $$\n(X):=\n_{\widehat X}(X)$$ is usually called the \emph{von Neumann radical} of $X$. Then $\n(X)=\{0\}$ (respectively, $\n(X)=X$) precisely when $\widehat X$ separates the points of $X$ (respectively, $\wh X=\{0\}$); in other words:
\begin{itemize}
\item[(i)] $X$ is MAP if and only if $\n(X)=\{0\}$;
\item[(ii)] $X$ is MinAP if and only if $\n(X)=X$.
\end{itemize}

\begin{remark}\label{remarkn}
Let $X$ be a topological abelian group and $\Gamma$ a subset of $\widehat X$.
\begin{itemize}
\item[(i)] If $\Gamma=\emptyset$, then $\n_\Gamma(X)=X$. 
\item[(ii)] If $\Gamma$ is countable, then $\n_{\Gamma}(X)$ is a characterized subgroup of $X$ (indeed, $\n_{\Gamma}(X)=s_\vs(X)$ for $\vs\in\wh{X}^\N$ such that each character in $\Gamma$ appears infinitely many times in $\vs$).
\end{itemize}
\end{remark}

For a given sequence $\vs\in\wh{X}^\N$, the \emph{support} $\Gamma_\vs=\lgr v_n :n\in\N\rgr$  of $\vs$ is the set of all characters appearing in $\vs$. We abbreviate the notation of the ${\Gamma_\vs}$-radical by writing
$$\n_\vs(X)\mathrel{\mathop:}=\n_{\Gamma_\vs}(X),$$
and we call this subgroup the \emph{$\vs$-radical} of $X$. 

\begin{lemma}\label{Rnv}\label{LnvGdelta}
Let $X$ be a topological abelian group and $\vs\in\widehat{X}^\N$. Then:
\begin{itemize}
\item[(i)] $\n_\vs(X)\le\sv(X)$;
\item[(ii)] $\n_\vs(X)$ is dually closed;
\item[(iii)] $\n_\vs(X)$ is characterized;
\item[(iv)] $\n_\vs(X)$ is closed and $\{0\}$ is $\G$ in $X/\n_\vs(X)$ (so, $\n_\vs(X)$ is $\G$); 
\item[(v)] $[X:\n_\vs(X)]\le\cc$.
\end{itemize}
\end{lemma}
\begin{proof}
(i) and (ii) are clear from the definitions and (iii) follows from Remark \ref{remarkn}(ii).

(iv) Let $\varphi:X\to\T^\N$ be defined by $\varphi(x)=(v_0(x),\dots,v_n(x),\dots)$ for every $x\in X$. Since $\{0\}$ is $\G$ in $\T^\N$, we conclude that  that $\{0\}$ is $\G$ in $X/\n_\vs(X)$. Moreover, $\ker\varphi=\varphi^{-1}(0)=\n_\vs(X)$, so $\n_\vs(X)$ is $\G$ in $X$. 

(v) Since $X/\n_\vs(X)$ is algebraically isomorphic to $\varphi(X)\leq\T^\N$ and $|\T^\N|=\cc$, we conclude that $[X:\n_\vs(X)]\le\cc$.
\end{proof}

\begin{rem}\label{RchIndex}
Let $X$ be a topological abelian group and $\vs\in\wh{X}^\N$. Then $\n_\vs(X)$ is closed and $\G$ in every group topology on $X$ that makes $v_n$ continuous for every $n\in\N$. In particular, $\n_\vs(X)$ is closed and $\G$ in every group topology on $X$ compatible with the topology of $X$, so in the Bohr topology of $X$. 
\end{rem}

Lemma \ref{LnvGdelta} gives a bound for the index of the characterized subgroups:

\begin{cor}\label{<c}
Every characterized subgroup of a topological abelian group $X$ has index at most $\cc$.  \end{cor}

\begin{proof}
Let $\vs\in\wh X^\N$. Since $\n_\vs(X)\le\sv(X)$ by Lemma \ref{Rnv}(i), hence $[X:\sv(X)]\le[X:\n_\vs(X)]\le\cc$ by Lemma \ref{Rnv}(v).
\end{proof}

The set $\Gamma_\vs$ can be partitioned as $$\Gamma_\vs=\Gi\dot\cup\Gz,$$ where
\begin{itemize}
  \item[(i)] $\Gi:=\lgr v_n\in \Gamma_\vs: v_n=v_m\ \text{for infinitely many}\ m\in\N\rgr$;
  \item[(ii)] $\Gz:=\Gamma_\vs\setminus \Gi$.		
\end{itemize}
In other words, $\Gi$ is the set of all characters appearing infinitely many times in $\vs$, while each character in its complement $\Gz$ appears finitely many times in $\vs$. Clearly, $\vs$ is a finitely many-to-one sequence if and only if $\Gamma_\vs^\infty=\emptyset$. 

\medskip
In case $\Gamma_{\vs^\infty}\ne \emptyset$, let $\vs^\infty$ be the largest subsequence of $\vs$ with $\Gamma_{\vs^\infty}=\Gamma_{\vs}^\infty$. Then clearly $s_{\vs^\infty}(X)=\n_{\vs^\infty}(X)$.

In case $\Gamma_{\vs}^0$ is finite, the subsequence $\vs^\infty$ of $\vs$ is cofinite, so $s_\vs(X) = s_{\vs^\infty}(X)$. In other words, one can safely replace 
$\vs$ by $\vs^\infty$. 
This is why \emph{from now on we shall always assume that}
\begin{equation}\label{dag}
\text{\emph{either}}\ \Gamma_{\vs^0}=\emptyset\ \text{\emph{or}}\ \Gamma_{\vs^0}\ \text{\emph{is infinite}}.
\end{equation}
If $\Gamma_{\vs}^0$ is infinite, we denote by $\vs^0$ the subsequence of $\vs$ such that $\Gamma_{\vs^0}=\Gz$. 
If $\Gamma_{\vs}^\infty\ne \emptyset$, we have the partition $$\vs=\vs^\infty\dot\cup\vs^0$$ of $\vs$ in the two subsequences $\vs^\infty$ and $\vs^0$.
Moreover, always $\Gamma_{\vs^0}^\infty=\emptyset$ and so, if $\Gamma_{\vs^0}$ is infinite (equivalently, $\Gamma_{\vs^0}\neq\emptyset$ by \eqref{dag}), 
$\vs^0$ is a finitely many-to-one sequence and $s_{\vs^0}(X)=s_{\mathbf w}(X)$, where $\ws$ is a one-to-one subsequence of $\vs^0$ such that $\Gamma_\ws=\Gamma_{\vs^0}$.

\medskip
We see now how we can obtain the subgroup of $X$ characterized by $\vs$ by considering separately the $\vs^\infty$-radical of $X$ and the subgroup of $X$ characterized by $\vs^0$.

\begin{lemma}\label{RKchar}
Let $X$ be a topological abelian group and $\vs\in\wh{X}^\N$ satisfying \eqref{dag}. 
\begin{itemize}
\item[(i)] If $\Gz = \emptyset$, then $\vs^\infty = \vs$, so $\Gamma^\infty \ne \emptyset$ and $\sv(X)=\n_{\vs^\infty}(X)$.
\item[(ii)]If $\Gz$ is infinite and $\Gi \ne \emptyset$, then $s_\vs(X)=s_{\vs^0}(X)\cap\n_{\vs^\infty}(X)$.
\end{itemize}
\end{lemma}
\begin{proof}
(i) Since $\Gz= \emptyset$, we have $s_\vs(X)=s_{\vs^\infty}(X)$ and, as observed above, $s_{\vs^\infty}(X)=\n_{\vs^\infty}(X)$. 

(ii) Since $\vs^\infty$ and $\vs^0$ are subsequences of $\vs$, it follows that $s_\vs(X)\\leq s_{\vs^0}(X)\cap\n_{\vs^\infty}(X)$. Let now $x\in s_{\vs^0}(X)\cap\n_{\vs^\infty}(X)$. Since both $\vs^\infty(x)\to0$ and $\vs^0(x)\to 0$, and since $\vs=\vs^\infty\dot\cup\vs^0$, we conclude that $\vs(x)\to 0$, that is, $x\in s_\vs(X)$. This concludes the proof.
\end{proof}

For $\vs=(v_n)\in\wh{X}^\N$ and $m\in\N$, let 
\begin{equation}\label{vmeq}
\vs_{(m)}:=(v_n)_{n\ge m}.
\end{equation}
Note that $\n_{\vs_{(m)}}(X)\leq \n_{\vs_{(m+1)}}(X)$ for every $m\in\N$.

\section{A hierarchy for characterized subgroups}\label{hsec}

The following definition introduces three specific types of characterized subgroups.

\begin{defin}\label{T-N-K}
Let $X$ be a topological abelian group. A subgroup $H$ of $X$ is:
\begin{itemize}
\item[(i)] \emph{$T$-characterized} if $H=\sv(X)$ where $\vs\in\wh{X}^\N$ is a non-trivial $T$-sequence;
\item[(ii)] \emph{$K$-characterized} if $H=\sv(X)$ for some finitely many-to-one sequence $\vs\in\wh{X}^\N$ (i.e., $\Gi=\emptyset$);
\item[(iii)] \emph{$N$-characterized} if $H=\n_\vs(X)$ for some $\vs\in\wh{X}^\N$.
\end{itemize}
\end{defin}

In analogy to Definition \ref{T-N-K}(i), we introduce the following smaller class of characterized subgroups (see also Problem \ref{TBPB}).

\begin{defin}\label{TB}
A subgroup $H$ of a topological abelian group $X$ is \emph{$TB$-characterized} if $H=\sv(X)$, where $\vs\in\wh{X}^\N$ is a non-trivial $TB$-sequence.
\end{defin}

Every $TB$-characterized subgroup is obviously $T$-characterized. Moreover, every $T$-characterized subgroup is also $K$-characterized, since every $T$-sequence contains no constant subsequences. The $N$-characterized subgroups are clearly closed, and they are always characterized as noted above.

On the other hand, proper dense characterized subgroups are $TB$-characterized by Lemma \ref{vv}(ii), so also $T$-characterized, and in particular $K$-characterized, but they are not $N$-characterized. We shall see below that closed (even open) subgroups need not be $K$-characterized in general. 

\smallskip
Denote by $\CH_K(X)$ (respectively, $\CH_N(X)$, $\CH_T(X)$, $\CH_{TB}(X)$) the family of all $K$-characterized (respectively, $N$-characterized, $T$-characterized, $TB$-characterized) subgroups of the topological abelian group $X$. Then we have the following strict inclusions
$$\CH_{TB}(X)\subsetneq \CH_T(X)\subsetneq \CH_K(X)\subsetneq \CH(X)\supsetneq \CH_N(X).$$ 

\bigskip
We start giving some basic properties that can be proved immediately.

\begin{cor}\label{NEW:lemmaK}
Let $X$ be a topological abelian group and $X_0$, $X_1$, $X_2$ subgroups of $X$ with $X_0\leq X_1\leq X_2$ and such that $X_1$ is dually embedded in $X_2$. 
\begin{itemize}
\item[(i)] If $X_0\in \CH_K(X_1)$ and $X_1\in \CH_K(X_2)$, then $X_0\in \CH_K(X_2)$. 
\item[(ii)] If $X_0\in \CH_N(X_1)$ and $X_1\in \CH_N(X_2)$, then $X_0\in \CH_N(X_2)$. 
\end{itemize}
\end{cor}
\begin{proof}
(i) It suffices to note that if in the proof of Proposition \ref{NEW:lemma}, $\vs$ is one-to-one and $\ws$ is one-to-one, then $\ws^*$ is finitely many-to-one.

(ii) It suffices to note that if in the proof of Proposition \ref{NEW:lemma}, $\Gamma_\vs=\Gamma_\vs^\infty$  and $\Gamma_\ws=\Gamma_\ws^\infty$, then also $\Gamma_{\ws^*}=\Gamma_{\ws^*}^\infty$. 
\end{proof}

By Lemma \ref{RKchar}, we have directly the following

\begin{cor}\label{NcapK}
Every characterized subgroup of a topological abelian group $X$ is the intersection of an $N$-characterized subgroup of $X$ and a $K$-characterized subgroup of $X$.
\end{cor}

The following stability property is clear for $N$-characterized subgroups, while it is not known for characterized subgroups.

\begin{lemma}
Countable intersections of $N$-characterized subgroups are $N$-characterized.
\end{lemma}

The next correspondence theorem was proved in \cite{DG13} for characterized subgroups of compact abelian groups. 

\begin{prop}\label{PcntrimgChar}
Let $X$ be a topological abelian group, $F$ a closed subgroup of $X$ and let $\pi:X\to X/F$ be the canonical projection. If $H$ is a characterized (respectively, $K$-characterized, $N$-characterized, $T$-characterized) subgroup of $X/F$, then $\pi^{-1}(H)$ is a characterized (respectively, $K$-characterized, $N$-characterized, $T$-characterized) subgroup of $X$.
\end{prop}
\begin{proof}
Let $\us=(u_n)\in\wh{X/F}^\N$ and consider $\wh{\pi}:\wh{X/F}\to\widehat X$. For every $n\in\N$ let $v_n=\wh{\pi}(u_n)=u_n\circ\pi\in F^\perp\leq\wh X$ and $\vs=(v_n)$. 

(i) Assume that $H=s_\us(X/F)$. Then $\pi^{-1}(H)=\sv(X)$, as $x\in\sv(X)$ if and only if $v_n(x)=u_n(\pi(x))\to 0$, and this occurs precisely when $\pi(x)\in H$.
 
(ii) Assume now that $H$ is $K$-characterized, that is, assume that $H=s_\us(X/F)$ and that $\Gamma_\us^\infty=\emptyset$. By (i), $\pi^{-1}(H)=\sv(X)$, and moreover, $\Gamma_{\vs}^\infty=\emptyset$.

(iii) If $H$ is $N$-characterized, then assume that $H=\n_{\us}(X/F)$. By (i), $\pi^{-1}(H)=\sv(X)$, and moreover $\pi^{-1}(H)=\n_\vs(X)$, since $v_n(x)=u_n(\pi(x))=0$ for every $n\in\N$ precisely when $\pi(x)\in H=\n_{\us}(X/F)$.

(iv) If $H$ is $T$-characterized, that is, if $H=s_\us(X/F)$ and $\us$ is a $T$-sequence, it remains to verify that $\vs$ is a $T$-sequence as well, since $\pi^{-1}(H)=s_\vs(X)$ by (i). Let $\tau$ be a Hausdorff group topology on $\wh{X/F}$, such that $u_n\to 0$ in $(\wh{X/F},\tau)$. By Lemma \ref{perp}(i), one can identify $\wh{X/F}$ with the subgroup $F^\perp$ of $\wh X$ by the algebraic isomorphism $\psi:\wh{X/F}\to \wh X$ defined by $\chi\mapsto \chi\circ\pi$. Let $\tau^*$ be the group topology on $\wh X$ having as a local base at $0$ the open neighborhoods of $0$ in $(\wh{X/F},\tau)$. Then $\tau^*$ is a Hausdorff group topology on $\wh{X}$ and $v_n\to 0$ in $(\wh X,\tau^*)$, as $v_n=\psi(u_n)\in F^\perp$ for every $n\in\N$  by definition.
\end{proof}

\begin{lemma}\label{aa}
Let $X$ be a topological abelian group and $H$ a subgroup of $X$ such that $\n(X)\leq H$. Then $H$ is characterized (respectively, $K$-characterized, $N$-characterized, $T$-characterized) if and only if $H/\n(X)$ is characterized (respectively, $K$-characterized, $N$-characterized, $T$-characterized).
\end{lemma}
\begin{proof}
Let $H=s_\vs(X)$ for some $\vs\in\wh X^\N$ and denote by $\pi:X\to X/\n(X)$ the canonical projection. For every $n\in\N$, since $\n(X)\leq \ker v_n$, the character $v_n$ factorizes as $v_n=u_n\circ \pi$, where $u_n\in\wh{X/\n(X)}$. Then $H/\n(X)=s_\us(X/\n (X))$. Viceversa, if $H/\n(X)=s_\us(X/\n(X))$ for some $\us\in\wh{X/\n(X)}^\N$, let $v_n=u_n\circ \pi$ for every $n\in\N$. Hence, $H=s_\vs(X)$. Moreover, $\vs$ is a finitely many-to-one sequence if and only if $\us$ is a finitely many-to-one sequence, so $H$ is $K$-characterized if and only if $H/\n(X)$ is $K$-characterized. Similarly, $\Gamma_\us^0$ is finite, precisely when $\Gamma_\vs^0$ is finite, hence $H$ is $N$-characterized if and only if $H/\n(X)$ is $N$-characterized.

It remains to check that $\vs$ is a $T$-sequence precisely when $\us$ is a $T$-sequence. This follows from the fact that the natural homomorphism $\wh{X/\n(X)} \to \wh{X}$ sending (the members of) $\us$ to (the members of) $\vs$ is an isomorphism, so certainly preserves the property of being a $T$-sequence. 
\end{proof}

The following lemma gives equivalent conditions for a subgroup to be characterized.

\begin{lemma}\label{TnewTeo}
Let $X$ be a topological abelian group and $H$ a subgroup of $X$. The following conditions are equivalent:
\begin{itemize}
\item[(i)] $H\in\CH(X)$;
\item[(ii)] there exists a closed subgroup $F$ of $X$ such that $F\le H$ and $H/F\in\CH(X/F)$;
\item[(iii)] there exists $\vs\in\wh{X}^\N$ such that for every closed $F\le\n_\vs(X)$ one has that $H/F=\su(X/F)$, where $\us=(u_n)$ and each $u_n$ is the factorization of $v_n$  through the canonical projection $\pi:X\to X/F$.
\end{itemize}
\end{lemma}
\begin{proof}
(iii)$\Rightarrow$(ii) Take $F=\n_\vs(X)$.

(ii)$\Rightarrow$(i) Since $F\le H$ one has $H=\pi^{-1}(H/F)$ and by Proposition \ref{PcntrimgChar} one can conclude.

(i)$\Rightarrow$(iii) Let $H=\sv(X)$ for some $\vs\in\wh X^\N$ and $F\le\n_\vs(X)$. Let $\pi: X\to X/F$ be the canonical projection. For every $n\in\N$, let $u_n$ be the character of $X/F$ defined by $\pi(x)\mapsto v_n(x)$. Then $u_n$ is well defined since $F\le\n_\vs(X)\le\ker v_n$. Hence, $u_n\circ\pi=v_n$ for every $n\in\N$, and $H/F=s_\us(X/F)$. Indeed, for every $h\in H$, we have $u_n(\pi(h))=v_n(h)\to0$ and hence $H/F\leq s_\us(X/F)$. Conversely, if $\pi(x)\in s_\us(X/F)$, then $v_n(x)=u_n(\pi(x))\to0$. Hence, $x\in H$ and so $\pi(x)\in H/F$.
\end{proof}

\section{Autocharacterized groups}\label{autosec}

 The following consequence of \cite[Proposition 2.5]{DG13} motivates the introduction of the notion of autocharacterized group (see Definition \ref{PA}).

\begin{prop}\label{CCTWiii}
Let $X$ be a compact abelian group. Then $\sv(X) = X$ for some $\vs\in\wh X^\N$ if and only if the sequence $\vs$ is eventually null.
\end{prop}
\begin{proof}
It is clear from \eqref{ev0} that $\sv(X) = X$ if $\vs$ is eventually null.
Assume now that $X=s_\vs(X)$ for some $\vs\in\wh X^\N$. By \cite[Proposition 2.5]{DG13}, being $s_\vs(X)$ compact, there exists $m\in\N$ such that $X=s_\vs(X)=\n_{\vs_{(m)}}(X)$, and so $v_n=0$ for all $n\geq m$. 
\end{proof}

If one drops the compactness, then the conclusion of Proposition \ref{CCTWiii} may fail, as shown in the next example.

\begin{exa}\label{EnotEvnull}
\ \vspace{-.5cm}\\
\begin{itemize}
\item[(i)] Let $N$ be an infinite countable subgroup of $\T$. As mentioned in the introduction, $N$ is characterized in $\T$, then $N=\sv(\T)$ for a non-trivial sequence $\vs\in\Z^\N$. If $\us=\vs\rs_N$, then $\us$ is non-trivial (since $N$ is dense in $\T$) and $\su(N)=N$.
\item[(ii)]  Let $X = \R$ and $\vs = (v_n)\in \widehat \R^\N$ such that $v_0=0$ and $v_n(x) = \pi(\frac{x}{n})\in \T$ for every $x \in \R$ and $n\in \N_+$. Obviously, $\sv(\R) = \R$, even if $\vs$ is non-trivial. 
\item[(iii)] Let $X = \Q_p$, where $p$ is a prime. For every $n\in\N$, let $v_n = p^n\in \widehat{\Q_p}$. Obviously, $\sv(\Q_p) = \Q_p$, even if $\vs$ is non-trivial. 
\end{itemize}
\end{exa}

Motivated by Proposition \ref{CCTWiii} and Example \ref{EnotEvnull}, we give the following

\begin{defin}\label{PA}
A topological abelian group $X$ is \emph{autocharacterized} if $X=\sv(X)$ for some non-trivial $\vs\in\wh{X}^\N$.
\end{defin}

Items (ii) and (iii) of Example \ref{EnotEvnull} show that $\R$ and $\Q_p$ are autocharacterized.

\begin{remark}\label{wlog11}
\begin{itemize}
\item[(i)] Let $X$ be an autocharacterized topological abelian group, so let $\vs\in\wh X^\N$ be non-trivial and such that $X=s_\vs(X)$. Then there exists a one-to-one subsequence $\us$ of $\vs$ such that $u_n\neq 0$ for every $n\in\N$ and $X=s_\us(X)$.

Indeed, if $\chi\in \Gamma^\infty_\vs$, then $X=s_\vs(X)\leq\ker\chi$ and so $\chi=0$; therefore, $\Gamma_\vs^\infty$ is either empty or $\{0\}$. Since $\vs$ is non-trivial, $\Gamma_\vs^0$ is infinite, hence $X=s_\vs(X)=s_{\vs^0}(X)$ by Lemma \ref{RKchar}(ii). Let $\us$ be the one-to-one subsequence of $\vs^0$ such that $\Gamma_\us=\Gamma_{\vs^0}$, therefore $X=s_\us(X)$.

\item[(ii)] The above item shows that autocharacterized groups are $K$-characterized subgroups of themselves.
But one can prove actually that they are $T$-characterized subgroups of themselves (indeed, $TB$-characterized subgroups of themselves, see \cite{DAGB}).
\end{itemize}
\end{remark}

\subsection{Basic properties of autocharacterized groups}

We start by a direct consequence of Lemma \ref{aa}:

\begin{lemma}\label{aacor}
Let $X$ be a topological abelian group and $H$ a subgroup of $X$ such that $\n(X)\leq H$. Then $X$ is autocharacterized if and only if $X/\n(X)$ is autocharacterized. 
\end{lemma}

The next proposition, describing an autocharacterized group in terms of the null sequences of its dual, follows from the definitions and Lemma \ref{vv}: 

\begin{prop}\label{prop:auotochar}
A topological abelian group $X$ is autocharacterized if and only if $(\widehat X,\sigma(\widehat X, X))$ has a non-trivial null sequence $\vs$ (in such a case $X = s_\vs(X)$). 
\end{prop}


In the next lemma we see when a subgroup of an autocharacterized group is autocharacterized, and viceversa.

\begin{lemma}\label{directsummand}\label{lemma:Aug2}
Let $X$ be a topological abelian group and $H$ a subgroup of $X$. 
\begin{itemize}
\item[(i)] If $X$ is autocharacterized and $H$ is dense in $X$, then $H$ is autocharacterized.
\item[(ii)] If $H$ is autocharacterized and one of the following conditions holds, then $X$ is autocharacterized:
\begin{itemize}
\item[(a)] $H$ is a topological direct summand of $X$;
\item[(b)] $H$ is open and has finite index.
\end{itemize}
\end{itemize}
\end{lemma}
\begin{proof}
(i) Let  $X = s_\vs(X)$ for $\vs\in\wh X$ with $v_n\neq 0$ for every $n\in\N$, and let $u_n = v_n\restriction_{H}\in\wh H$. Then each $u_n$ is non-zero and $H=s_\us(H)$.

(ii) Let $H=\sv(H)$ for some $\vs\in\wh H^\N$ with $v_n\neq0$ for every $n\in\N$.

(a) Let $X = H \times Z$. For every $n\in \N$, let $u_n$ be the unique character of $X$ that extends $v_n$ and such that $u_n$ vanishes on $Z$. Then $u_n \ne 0$ for every $n\in\N$, and $X=s_\us(X)$.

(b) Arguing by induction, we can assume without loss of generality that $[X:H]=p$ is prime. 
Let $X = H +\langle x\rangle$ with $x\not \in H$ and $px\in H$. If $px=0$, then $H$ is an open direct summand of $X$, so $H$ is also a topological direct summand of $X$, hence item (a) applies. 
Assume now that $px\ne 0$, and let $a_n = v_n(px)$ for every $n\in\N$. If $a_n = 0$ for infinitely many $n$, extend $v_n$ to $u_n\in\wh X^\N$ for those $n$ by letting $u_n(x) = 0$. Then obviously the sequence $\us$ obtained in this way is not eventually null and $X = s_\us(X)$, so $X$ is autocharacterized. 
Assume now that $a_n \ne 0$ for infinitely many $n$; for those $n$, pick an element $b_n\in \T$ with $pb_n = a_n$ and extend $v_n$ by letting $u_n(h + kx) = v_n(h) + kb_n$. Let $w_n = pu_n$. Then $w_n(x) = a_n\ne 0$, so $\ws$ is not eventually null. Moreover, $X = s_\us(X)$ as $pX \leq H$.
\end{proof}

\begin{lemma}\label{PnPAnChar}
Let $X$ be a topological abelian group and $\vs\in\wh{X}^\N$. If $F$ is a subgroup of $X$ such that $F\le\sv(X)$ and $F$ is not autocharacterized, then $F\le\n_{\vs_{(m)}}(X)$ for some $m\in\N$. 
\end{lemma}
\begin{proof}
Let $u_n=v_n\restriction_F$ for every $n\in\N$ and $\us=(u_n)\in\wh{F}^\N$.
Then $F = \su(F)$, so the sequence $\us$ must be eventually null. Let $m\in\N$ such that $u_n =0$ for every $n\ge m$. Therefore, 
$F \leq\n_{\vs_{(m)}}(X)$. 
\end{proof}

The following consequence of Lemma \ref{PnPAnChar} is a generalization of \cite[Lemma 2.6]{DG13} where the group $X$ is compact.

\begin{cor}
Let $X$ be a topological abelian group, $F$ and $H$ subgroups of $X$ such that $F$ is compact and $F\le H$. Then $H/F\in\CH(X/F)$ if and only if $H\in\CH(X)$.
\end{cor}
\begin{proof}
Denote by $\pi:X\to X/F$ the canonical projection.
If $H/F$ is a characterized subgroup of $X/F$ then $H=\pi^{-1}(H/F)$ is a characterized subgroup of $X$ by Proposition \ref{PcntrimgChar}. 
Assume now that $H=s_\vs(X)$ for some $\vs\in\wh X^\N$. Since $F$ is compact, Proposition \ref{CCTWiii} implies that $F$ is not autocharacterized. By Lemma \ref{PnPAnChar}, $F$ is contained in $\n_{\vs_{(m)}}(X)$ for a sufficiently large $m\in\N$. Let $\pi':X/F\to X/\n_{\vs_{(m)}}(X)$ be the canonical projection and $q=\pi'\circ\pi$, then $q^{-1}(H/\n_{\vs_{(m)}}(X))=H=\sv(X)=s_{\vs_{(m)}}(X)$. Therefore, one deduces from Lemma \ref{TnewTeo} that $H/\n_{\vs_{(m)}}(X)$ is a characterized subgroup of $X/\n_{\vs_{(m)}}(X)$. Hence, by Proposition \ref{PcntrimgChar}, $H/F=(\pi')^{-1}({H/\n_{\vs_{(m)}(X)}})\in\CH(X/F)$.
\end{proof}

The argument of the above proof fails in case $F$ is not compact. For example, take $F=H=X=N$, where $N$ is as in Example \ref{EnotEvnull}; then one cannot conclude that  $\vs\restriction_F$ is eventually null, and hence that $F$ is contained in $\n_{\vs_{(m)}}(X)$.

\subsection{Criteria describing autocharacterized groups}

Here we give two criteria for a group to be autocharacterized. We start below with a criteria for locally compact abelian groups, while a general one, in terms of  the Bohr compactification, will be given at the end of the section. 
 
\medskip
We established in Proposition \ref{CCTWiii} that no compact abelian group is autocharacterized, now
we prove in Theorem \ref{compactnonac} that this property describes the compact abelian groups within the larger class of all locally compact abelian groups. This follows easily from  Lemma \ref{directsummand}(ii) for the locally compact abelian groups that contain a copy of $\R$, while the general cases requires the following deeper argument. 

\begin{teo}\label{compactnonac}
If $X$ is a locally compact abelian group, then $X$ is autocharacterized if and only if $X$ is not compact. 
\end{teo}
\begin{proof} 
If $X$ is autocharacterized, then $X$ is not compact according to Proposition \ref{CCTWiii}. Assume now that $X$ is not autocharacterized. 
Then by Fact \ref{sigma=+} and Proposition \ref{prop:auotochar} the dual $\wh X$ has no non-trivial null sequences in its Bohr topology. 
But since $\wh X$ is locally compact, it has the same null sequences as its Bohr modification $\wh X^+$. Therefore, $\wh X$ is
a locally compact group without non-trivial null sequences. We have to conclude that $X$ is compact.

This follows from the conjunction of several facts. The first one is the deep fact that non-discrete locally compact abelian groups have non-trivial null sequences. (This follows, in turn, from that fact that a non-discrete locally compact abelian group either contains a line $\R$, or an infinite compact subgroup. Since
compact groups are dyadic compacts, i.e., continuous images of Cantor cubes, they have non-trivial null sequences.)
Now we can conclude that the locally compact group $\wh X$ is discrete. It is a well known fact that this implies compactness of $X$. 
\end{proof}

\begin{rem}
An alternative argument to prove that non-discrete locally compact abelian groups have non-trivial null sequences is based on a theorem by Hagler, Gerlits and Efimov (proved independently also by Efimov in \cite{Efimov}). It states that every infinite compact group $K$ contains a copy of the Cantor cube $\{0,1\}^{w(K)}$, which obviously has a plenty of non-trivial null sequences.
\end{rem}

In order to obtain a general criterion describing autocharacterized groups we need another relevant notion in the theory of characterized subgroups: 

\begin{defin}\cite{DMT05}
Let $X$ be a topological abelian group and $H$ a subgroup of $X$. Let
$$\g_X(H)=\bigcap\lgr \sv(X): \vs\in\wh{X},\ H\le\sv(X)\rgr.$$
A subgroup $H$ of $X$ is said to be:
\begin{itemize}
\item[(i)] \emph{$\g$-dense} if $\g_X(H)=X$;
\item[(ii)] \emph{$\g$-closed} if $\g_X(H)=H$.
\end{itemize}
\end{defin}

We write simply $\g(H)$ when there is no possibility of confusion. Clearly, $\g(H)$ is a subgroup of $X$ containing $H$. Moreover, $\g(\{0\})$ is the intersection of all characterized subgroups of $X$ and $\g(\{0\})\leq\n(X)$.

\begin{rem}
Let $X=(X,\tau)$ be a topological abelian group and $H$ a subgroup of $X$. 
\begin{itemize}
\item[(i)]  If $X_1$ is another topological abelian group and $\phi:X\to X_1$ a continuous homomorphism, then $\phi(\g_X(H))\leq \g_{X_1}(\phi(H))$ (see \cite[Proposition 2.6]{DMT05}).
\item[(ii)] Moreover, $\g_X(H)\leq \overline H^{\tau^+}$. Indeed, $\overline H^{\tau^+}=\bigcap\{\ker\chi:\chi\in\wh X,\ H\leq\ker\chi\}$ and $\ker\chi=\n_\vs(X)=s_\vs(X)$ for $\vs\in\wh X^\N$ with $\Gamma_\vs=\Gamma_\vs^\infty=\{\chi\}$ (i.e., $\vs$ is the constant sequence given by $\chi$).

Item (i) says, in terms of \cite{Dik06,DTho}, that $\g$ is a closure operator in the category of topological abelian groups. The inclusion $\g_X(H)\leq \overline H^{\tau^+}$ says that $\g$ is finer than the closure operator defined by $H \mapsto  \overline H^{\tau^+}$. 
\item[(iii)] If $H$ is dually closed, then $H$ is $\g$-closed by item (ii).
\item[(iv)] If $(X,\tau)$ is a locally compact abelian group, then every closed subgroup of $(X,\tau)$ is dually closed, and so (ii) implies that $\g(H)\leq\overline H^\tau$ for every subgroup $H$ of $X$. Therefore, $\g$-dense subgroups are also dense in this case.
\item[(v)] The inclusion $\g(H) \leq \overline H$ may fail if the group $H$ is not MAP (e.g., if $X$ is MinAP, then $\g(H) = X$ for every $H$, while $X$ may have proper closed subgroups). 
\end{itemize}
\end{rem}

The next result shows that the autocharacterized precompact abelian groups are exactly the dense non-$\mathfrak g$-dense subgroups of the compact abelian groups.

\begin{teo}\label{acprec}
Let $X$ be a precompact abelian group. The following conditions are equivalent:
\begin{itemize}
\item[(i)] $X$ is autocharacterized;
\item[(ii)] $X$ is not $\mathfrak g$-dense in its completion $\widetilde X$.
\end{itemize}
\end{teo}
\begin{proof}
(ii)$\Rightarrow$(i) Assume that $X$ is not $\mathfrak g$-dense in $K:=\widetilde X$. Then there exists a sequence $\vs\in\wh K^\N$ such that $X\leq s_\vs(K)<K$. By Proposition \ref{CCTWiii} (see also Remark \ref{wlog11}), we may assume without loss of generality that $v_n\neq 0$ for every $n\in\N$. Let $u_n=v_n\restriction_X$ for every $n\in\N$. Since $X$ is dense in $K$, clearly $u_n\neq0$ for every $n\in\N$. Moreover, $X=s_\us(X)$, hence $X$ is autocharacterized. 

(i)$\Rightarrow$(ii) Suppose that $X$ is autocharacterized, say $X = s_\us(X)$ for $\us\in\widehat X^\N$ such that $u_n\neq 0$ for every $n\in\N$. For ever $n\in\N$, let $v_n\in\wh K$ be the extension of $u_n$ to $K$. Then $X \leq s_\vs(K) < K$ by Proposition \ref{CCTWiii}, so $X$ is not $\mathfrak g$-dense in $K$.
\end{proof}

If $X$ is a MAP abelian group, then $\tau^+$ is precompact and the Bohr compactification of $X$ is $r_X:X\to bX$, where $bX$ is the completion of $(X,\tau^+)$ and $r_X$ is an injective homomorphism.
If $X$ is not MAP, then $\n(X)\neq\{0\}$. Consider the quotient $X/\n(X)$, which is a MAP group. Then take the Bohr compactification $r_{X/\n(X)}:X/\n(X)\to b(X/\n(X))$. The Bohr compactification of $X$ is $r_X:X\to bX$, where $bX:=b(X/\n(X))$ and $r_X=r_{X/\n(X)}\circ \pi$, where $\pi:X\to \n(X)$ is the canonical projection.

\begin{cor}\label{Teo:PA1}
Let $X$ be a MAP abelian group. The following conditions are equivalent: 
\begin{itemize}
\item[(i)] $X$ is autocharacterized;
\item[(ii)] $X$ is not $\mathfrak g$-dense in $bX$. 
\end{itemize}
\end{cor}
\begin{proof}
Since $X$ is MAP, $X$ embeds in $bX$.
By Lemma \ref{compatible}, $X$ and $X^+$ have the same characterized subgroups. Moreover, $X^+$ is precompact and by definition $bX$ is the completion of $X^+$. Then it suffices to apply Theorem \ref{acprec}.
\end{proof}

\begin{teo}\label{autocarcar}
Let $X$ be a topological abelian group. The following conditions are equivalent: 
\begin{itemize}
\item[(i)] $X$ is autocharacterized;
\item[(ii)] $r_X(X)$ is not $\mathfrak g$-dense in $bX$.
\end{itemize}
\end{teo}
\begin{proof}
Since $X$ is autocharacterized precisely when $X/\n(X)$ is autocharacterized by Lemma \ref{aacor}, apply Corollary \ref{Teo:PA1} to conclude.
\end{proof}

%

\section{$K$-characterized subgroups}\label{Ksec}

We start by recalling the following result from \cite{CTT93}: if $X$ is a compact abelian group and $\vs\in\widehat{X}^\N$ has a one-to-one subsequence, then $\sv(X)$ has Haar measure $0$ in $X$. In our terms, it reads as follows: 

\begin{lemma}\emph{\cite[Lemma 3.10]{CTT93}}\label{LCTW}\label{CKnotOpen}
If $X$ is a compact abelian group and $H \in \CH_K(X)$, then $H$ has Haar measure $0$ (hence, $[X:H]$ is infinite).
In particular, no open subgroup of $X$ is $K$-characterized. 
\end{lemma}

Lemma \ref{LCTW} cannot be inverted, take for example the constant sequence $\us=(1)$ in $\wh{\T}^\N$. 

\begin{rem}\label{m0}
If $X$ is a connected compact abelian group, then the conclusion of Lemma \ref{LCTW} holds for all non-trivial sequences in $\wh{X}$, since every measurable proper subgroup $H$ of $X$ has measure $0$ (indeed, $X$ is divisible, so the proper subgroup $H$ of $X$ has infinite index, hence the measure of $H$ must be $0$, as $X$ has measure $1$).
\end{rem}


\begin{exa} 
Here we provide examples of non-autocharacterized non-compact abelian groups. 
\begin{itemize}
\item[(i)] A relatively simple example can be obtained by taking a dense non-measurable subgroup $X$ of a connected compact abelian group $K$. Since we intend to deduce that $X$ is non-autocharacterized by using Theorem \ref{acprec}, we have to check that $X$ is $\mathfrak g$-dense in $K$. Indeed, every measurable proper subgroup of $K$ has measure $0$ as noted in Remark \ref{m0}, therefore every proper characterized (hence, every non-$\mathfrak g$-dense) subgroup of $K$ has measure $0$. Therefore, $X$ is not contained in any proper characterized subgroup of $K$, i.e., $X$ is $\mathfrak g$-dense in $K$. 
\item[(ii)] More sophisticated examples of $\mathfrak g$-dense subgroups of metrizable compact abelian groups were given in \cite{BDMW03} (under the assumption of Martin Axiom) and in  \cite{HK05} (in ZFC). These groups have the additional property of being of measure zero (so that the above elementary argument cannot be used to verify the $\mathfrak g$-density). 
\end{itemize}
\end{exa}

\subsection{When closed subgroups of infinite index are $K$-characterized}

The next theorem gives a sufficient condition (see item (iii)) for a closed subgroup of infinite index $H$ to be $K$-characterized. 
This condition implies, as a by-product, that $H$ is also $N$-characterized.

The easier case of open subgroups will be tackled in Theorem \ref{Kcar=car}, by applying Theorem \ref{T1122}.

\begin{teo}\label{T1122} 
Let $X$ be a topological abelian group and $H$ a closed subgroup of $X$ of infinite index. The following conditions are equivalent: 
\begin{itemize}
\item[(i)] $H\in\CH_K(X)\cap \CH_N(X)$;
\item[(ii)] $H\in\CH_N(X)$;
\item[(iii)] $X/H$ is MAP and $(\wh{X/H},\sigma(\wh{X/H},X/H))$ is separable.
\end{itemize}
\end{teo}
\begin{proof} 
(i)$\Rightarrow$(ii) is obvious.

(ii)$\Rightarrow$(iii) Since $H$ is $N$-characterized, then $H$ is dually closed by Lemma \ref{Rnv}(ii), that is, $X/H$ is MAP.
Let $H=\n_\vs(X)$ for some $\vs\in\wh X^\N$ and let $\pi:X\to X/H$ be the canonical projection. 
Since $\ker \pi = H \leq \ker v_n$ for every $n\in \N$, one can factorize $v_n: X \to \T$ through $\pi$, i.e., 
write $v_n =  \bar v_n\circ \pi$ for appropriate $\bar v_n\in \wh{X/H}$. It remains to verify that $D=\{\bar v_n:n\in\N\}$ is dense in $(\wh{X/H},\sigma(\wh{X/H},X/H))$. To this end, let $\bar y=\pi(y)\in X/H$; if $\xi_{\bar y}(D)=\{0\}$, then $\bar v_n(\bar y)=v_n(y)=0$ and so $y\in H$, that is, $\bar y=0$.

(iii)$\Rightarrow$(i) Let $Y = X/H$ equipped with the quotient topology. By hypotheses, $Y$ is infinite and MAP, while $\widehat Y$ 
 is an infinite topological abelian group with a  countably infinite dense subgroup $D$. According to Proposition \ref{PcntrimgChar} applied to the canonical projection $\pi:X \to Y$, it suffices to prove that $\{0\}$ is a $K$-characterized subgroup of $Y$. 
Let $D=\{v_n: n\in \N\}$ be a one-to-one enumeration of $D$ and $\vs=(v_n)$. 
To prove that $\sv(Y) = \{0\}$, we have to show that for every non-zero $y\in Y$ there exists a neighborhood $U$ of $0$ in $\T$ such that $v_n(y)\not \in U$ for infinitely many $n\in \N$. Actually, we show that $U = \T_+$ works for all non-zero $y\in Y$. In fact, for every $y\in Y\setminus\{0\}$ one has that $N_y:=\{d(y):d\in D\}=\{v_n(y): n\in \N\}$ is a non-trivial subgroup of $\T$, as $Y$ is MAP and $y\ne 0$.  

Let $y\in Y\setminus\{0\}$. If $N_y$ is infinite, then $N_y$ is dense in $\T$, so $N_y\setminus U$ is infinite and we are done. Now consider the case when $N_y$ is finite.
As $N_y \ne \{0\}$ and $U$ contains no non-trivial subgroups of $\T$, there exists $a\in N_y$ such that $a\not \in U$. Then the map $f_y:D \to \T$ defined by $f_y(d) = d(y)$ is a homomorphism with $f_y(D) =N_y$ finite. Therefore, $K:= \ker f_y$ is a finite-index subgroup of $D$, so $K$ is infinite. Let $a = v_m(y)$ for some $m \in \N$. Then $v_m + K = \{d\in D: d(y) = a\}$ is infinite as well. This means that $v_n(y)=a\not \in U$ for infinitely many $n$ (namely, those $n$ for which $v_n \in v_m + K $). Therefore, $v_n(y) \not \to 0$ and so $y \not \in s_\vs(Y)$. 

Finally, let us note that the above argument shows also that $H$ is $N$-characterized as obviously $H\leq\n_\vs(X)$.
\end{proof}

The following is an obvious consequence of Theorem \ref{T1122}.

\begin{cor}\label{N->K}
Let $X$ be a topological abelian group and $H$ a closed subgroup of $X$ of infinite index. If $H\in\CH_N(X)$, then $H\in\CH_K(X)$.
\end{cor}

Next we rewrite Theorem \ref{T1122} in the case of locally compact abelian groups.

\begin{cor}\label{T12}
Let $X$ be a locally compact abelian group and $H$ a subgroup of $X$. Then $H\in\CH_N(X)$ if and only if $H$ is closed and $\wh{X/H}$ is separable.
\end{cor}
\begin{proof} 
As both conditions imply that $H$ is closed, we assume without loss of generality that $H$ is closed.
Since $X/H$ and $\wh{X/H}$ are locally compact abelian groups, $X/H$ is MAP and the Bohr topology on $\wh{X/H}$ coincides with $\sigma(\wh{X/H},X/H)$ by Fact \ref{sigma=+}, so the separability of $\wh{X/H}$ is equivalent to the separability of $\wh{X/H}^+$ by Fact \ref{lcafact}. If $H$ has infinite index in $X$, apply Theorem \ref{T1122} to conclude. If $H$ has finite index in $X$, then the equivalence is trivially satisfied; indeed, $H$ is a finite intersection of kernels of characters, so it is $N$-characterized, and $\wh{X/H}$ is finite, so separable.
\end{proof}

As a consequence of Theorem \ref{T1122} we find a sufficient condition for an open subgroup of infinite  index to be $K$-characterized:

\begin{teo}\label{Kcar=car}
Let $X$ be a topological abelian group and $H$ an open subgroup of $X$ of infinite index. Then the following conditions are equivalent:
\begin{itemize}
\item[(i)] $H\in\CH(X)$;
\item[(ii)] $H\in\CH_K(X)$;
\item[(iii)] $[X:H]\leq\cc$;
\item[(iv)] $\wh{X/H}$ is separable.
\end{itemize}
\end{teo}
\begin{proof}
(ii)$\Rightarrow$(i) is clear and (i)$\Rightarrow$(iii) is Corollary \ref{<c}.

(iii)$\Rightarrow$(iv) Since $\wh{X/H}$ is a compact abelian group of weight at most $\cc$, it is separable by Lemma \ref{HMP}.

(iv)$\Rightarrow$(ii) As $[X:H]$ is infinite, we can apply Theorem \ref{T1122} to conclude that $H$ is $K$-characterized.
\end{proof}

The following is another direct consequence of Theorem \ref{T1122}.

\begin{cor}\label{CcomKchar}
If $X$ is a metrizable compact abelian group, then every closed non-open subgroup of $X$ is $K$-characterized.   
\end{cor}

\subsection{When closed subgroups of finite index are $K$-characterized}

We start by giving the following useful technical lemma.

\begin{lemma}\label{technical}
Let $X$ be a topological abelian group and $H$ an open subgroup of $X$ such that $X=H+\langle x\rangle$ for some $x\in X\setminus H$. If $H$ is autocharacterized, then $H\in \CH_K(X)$.
\end{lemma}
\begin{proof}
Let $\us\in\wh H^\N$ such that $H=s_\us(H)$. By Remark \ref{aa} we can assume that $\us$ is one-to-one and that $u_n\neq0$ for every $n\in\N$.

Assume first that $H\cap\langle x\rangle=\{0\}$. If $o(x)$ is infinite, then fix an irrational number $\alpha\in\R$ and for every $n\in\N$ let $v_n(x)=\alpha+\Z$ and $v_n(h)=u_n(h)$ for every $h\in H$. If $o(x)=k$ is finite, then for every $n\in\N$ let $v_n(x)=\frac{1}{k}+\Z$ and $v_n(h)=u_n(h)$ for every $h\in H$. In both cases, it is straightforward to prove that $H=s_\vs(X)$. Moreover, since $\us$ is one-to-one, then also $\vs$ is one-to-one. 

Suppose now that $H\cap\langle x\rangle=\langle mx\rangle$ for some $m\in\N$, with $m\geq 1$. As  $x\not \in H$, one has $m\geq 2$. 

For every $n\in\N$, let $a_n=u_n(mx)\in\T$. Since $u_n(mx)\to 0$, there exists $n_0$ such that $\Vert a_n\Vert<\frac{1}{m^2}$ for every $n\geq n_0$. As $s_{\us_{(n_0)}}(H)=H$, we shall assume for simplicity that $\Vert a_n\Vert<\frac{1}{m^2}$ every $n\in\N$. 

\begin{claim}\label{claim}
For every $a\in \T$ with $\Vert  a\Vert< \frac{1}{m^2}$, there exists $b\in \T$
such that 
\begin{equation}\label{bn}
m b=a\ \text{and}\ \Vert k b\Vert>\frac{1}{m^2}\ \text{for every}\ k\in\N,\ 1\leq k< m.
\end{equation}
\end{claim}
\begin{proof}
We tackle the problem in $\R$, that is, identifying $\T$ with $[0,1)$. 
First assume that $0\leq a<\frac{1}{m^2}$ and let 
$$b=\frac{a}{m}+\frac{1}{m}.$$
Then $m b=a+1\equiv_\Z a$ and $\frac{1}{m}\leq b\leq \frac{2}{m}$. 
Let now $k\in\N$ with $1\leq k\leq m-1$, then 
\begin{equation}\label{ac}
\frac{1}{m^2} < \frac{k}{m}\leq k b =k \frac{a}{m}+\frac{k}{m}< \frac{m-1}{m^2}+\frac{m-1}{m}=1-\frac{1}{m^2}.
\end{equation}
Therefore, $\Vert k b\Vert>\frac{1}{m^2}$. This establishes \eqref{bn} in the current case. 

It remains to consider the case $\frac{m^2-1}{m^2}< a <1$. Let $a^* = 1-a$, i.e., $a^* = -a$ in $\T$. Then obviously $\Vert  a^* \Vert < \frac{1}{m^2}$ and $0\leq a^*<\frac{1}{m^2}$. Hence, by the above case applied to $a^*$, there exists $b^*\in \T$ satisfying \eqref{bn} with $-a$ in place of $a$ (i.e., $mb^* = -a$). Let $b = -b^* \in \T$. Then \eqref{bn} holds true for $b$ and $a$, as $ \Vert k (-b)\Vert =  \Vert k b\Vert$ for every $k\in\N$ with $1\leq k<m$. 
\end{proof}

For every $n\in\N$, apply Claim \ref{claim} to $a_n$ to get $b_n$ as in \eqref{bn}, then define $v_n:X\to \T$ by letting $v_n(x)=b_n$ for every $n\in\N$ and $v_n(h)=u_n(h)$ for every $h\in H$. As $u_n(mx) = v_n(mx) = mv_n(x) = mb_n = a_n$, this definition is correct. Moreover, since $H$ is open in $X$, $v_n\in\wh X$. Since $\us$ is one-to-one, then also $\vs=(v_n)$ is one-to-one.

We show that 
\begin{equation}\label{iffeq}
v_n(kx)\to0\ \text{for}\ k\in\N\ \text{if and only if}\ k\in m\Z.
\end{equation}
In fact, if $k=k'm$ for some $k'\in\N$, $$v_n(kx)=k'v_n(mx)=k'a_n\to0.$$ Viceversa, assume that $k=k'm+r$, where $k'\in\N$ and $1\leq r\leq m-1$. Then $$v_n(kx)=k'v_n(mx)+rv_n(x)=k'a_n+rb_n\not\to0,$$ since $k'a_n\to 0$ and $\Vert rb_n\Vert\geq \frac{1}{m^2}$ by \eqref{bn}. 

We deduce finally that $H=s_\vs(X)$.  Indeed, $H=s_\us(H)\leq s_\vs(X)$, so it remains to prove that $s_\vs(X)\leq H$. 
To this end, let $y\in X\setminus H$, that is, $y=h+kx$ for some $h\in H$ and $k\in\N$ with $1\leq k<m$. Then $$v_n(y)=v_n(h+kx)=u_n(h)+v_n(kx).$$
Since $h\in H=s_\us(H)$, that is, $u_n(h)\to 0$, while $v_n(kx)\not\to0$ by \eqref{iffeq}, we conclude that $v_n(y)\not\to0$, that is, $y\not\in s_\vs(X)$.
Hence, $H=s_\vs(X)$.
\end{proof}

Every open finite-index subgroup is a finite intersection of kernels of characters, so it is $N$-characterized. In the next theorem we describe when a proper open finite-index subgroup is $K$-characterized. 

\begin{teo}\label{PtimesKchar:new}
Let $X$ be a topological abelian group and $H$ a proper open subgroup of $X$ of finite index. Then $H\in \CH_K(X)$ if and only if $H$ is autocharacterized. 
\end{teo}
\begin{proof} 
Assume that $H\in \CH_K(X)$. We can write $H =\su(X)$ for $\us\in\wh X^\N$ one-to-one.  Let $v_n = u_n \restriction_H$ for every $n\in\N$. Then the map $u_n\mapsto v_n$ is  finitely many-to-one, as $X/H$ is finite. Therefore, $\vs=(v_n)$ is finitely many-to-one. Obviously, $H =\sv(H)$, so $H$ is autocharacterized. 

Now assume that $H$ is autocharacterized. Since $H$ has finite index in $X$, there exist $x_1,\ldots, x_n\in X$ such that $X=H+\langle x_1,\ldots,x_n\rangle$ and that,  letting $X_i :=H+\langle x_1, \ldots, x_i \rangle$ for $i = 1,\ldots, n$ and $X_0:=H$, the subgroup $X_{i-1}$ is a proper subgroup of $X_i$  for $i = 1,\ldots, n$.  We shall prove by induction on $i= 1,\ldots, n$, that 
\begin{equation}\label{Laast:eq}
H \in\CH_K(X_{i}) .
\end{equation}
As $X = X_n$, this will give $H\in \CH_K(X)$, as desired. 

Before starting the induction, we note that according to Lemma \ref{lemma:Aug2}(ii), all subgroups $X_{i}$,  for $i = 1,\ldots, n$, are autocharacterized, as each $X_{i-1}$ is open in $X_i$. 
For $i=1$, the assertion in \eqref{Laast:eq} follows from Lemma \ref{technical}. Assume that $1< i \leq n$ and \eqref{Laast:eq} holds true for $i -1$, i.e., $H\in \CH_K(X_{i-1})$. Since  $X_{i-1}$ is open in $X_i$, again Lemma \ref{technical} applied to $X_{i} = X_{i-1} + \langle  x_i \rangle$ gives that $X_{i-1}\in\CH_K(X_{i})$. As $H\in \CH_K(X_{i-1})$ by our inductive hypothesis, we conclude with  Corollary \ref{NEW:lemmaK}(i) that $H \in\CH_K(X_i)$. 
\end{proof}

\subsection{Further results on $K$-characterized subgroups}

The next corollary resolves an open question from \cite{Impi}:

\begin{cor}\label{Kcar=card}
Let $X$ be an infinite discrete abelian group and $H$ a subgroup of $X$. The following conditions are equivalent:
\begin{itemize}
\item[(i)] $H\in\CH(X)$;
\item[(ii)] $H\in\CH_K(X)$; 
\item[(iii)] $[X:H]\leq\cc$.
\end{itemize}
\end{cor}
\begin{proof}
(ii)$\Rightarrow$(i) is clear and (i)$\Rightarrow$(iii) is Corollary \ref{<c}.

(iii)$\Rightarrow$(ii) If $[X:H]$ is infinite, then $H\in\CH(X)$ by Theorem \ref{Kcar=car}. So, assume that $[X:H]$ is finite, then $H$ is infinite, and hence $H$ is autocharacterized by Theorem \ref{compactnonac}; therefore, $H\in\CH_K(X)$ by Theorem \ref{Kcar=car}.
\end{proof}

 We give now sufficient conditions for a non-closed characterized subgroup to be $K$-characterized.

\begin{teo}\label{CHK}
Let $X$ be a topological abelian group and $H\in \CH(X)$ a non-closed subgroup of $X$ such that:
\begin{itemize}
\item[(i)] $X/\overline H$ is MAP and $(\wh{X/\overline H},\sigma(\wh{X/\overline H},X/\overline H))$ is separable; 
\item[(ii)] if $1< [X:\overline{H}]<\omega$, then $\overline{H}$ is autocharacterized.
\end{itemize}
Then $H\in \CH_K(X)$. 
\end{teo}
\begin{proof} 
As $H \ne \overline H$ is dense in $\overline H$ and obviously $H \in \CH(\overline H)$, we deduce that $H\in\CH_K(\overline H)$, as dense characterized subgroups are $TB$-characterized by Lemma \ref{vv}(ii). If $\overline H=X$, we are done. So assume that $\overline H$ is proper.
 
Our aim now is to apply Corollary \ref{NEW:lemmaK}, so we need to check that $\overline H \in \CH_K(X)$. 
If $\overline H$ has finite index in $X$, then $\overline H$ is autocharacterized by hypothesis and so Theorem \ref{PtimesKchar:new} yields $\overline H\in\CH_K(X)$.
If $\overline H$ is has infinite index in $X$, then $\overline H\in \CH_K(X)$ by Theorem \ref{T1122}.
\end{proof}

\begin{cor}\label{CHK:new}
Let $X$ be a divisible topological abelian group and $H\in \CH(X)$ a non-closed subgroup of $X$ such that 
$\overline H$ is dually closed. If $(\wh{X/\overline H},\sigma(\wh{X/\overline H},X/\overline H))$ is separable, then $H\in \CH_K(X)$. 
In particular, $H\in \CH_K(X)$ whenever $\wh{X/\overline H}$ is separable. 
\end{cor}
\begin{proof} 
The first part of our hypothesis entails that $X/\overline H$ is MAP. Moreover, divisible topological abelian groups have no proper closed subgroup of finite index. Therefore, the first assertion follows directly from Theorem \ref{CHK}.  

The topology $\sigma(\wh{X/\overline H},X/\overline H)$ of the dual $\wh{X/\overline H}$ is coarser than the compact-open topology of $\wh{X/\overline H}$, so that separability of $\wh{X/\overline H}$ yields separability of $(\wh{X/\overline H},\sigma(\wh{X/\overline H},X/\overline H))$. Hence, the second assertion can be deduced from the first one. 
\end{proof}

In the case of connected locally compact abelian groups one obtains the following stronger conclusion: 

\begin{cor}
Let $X$ be a connected locally compact abelian group. Then 
$$\CH_K(X)=\begin{cases} \CH(X), \mbox{ if $X$ is not compact} \\
\CH(X)\setminus \{X\}, \mbox{ if $X$ is compact}.\end{cases}$$
\end{cor}
\begin{proof}  
The group $X$ is divisible, as connected locally compact abelian groups are divisible. 

Let $H \in \CH(X)$. Our next aim will be to check that  $\wh{X/\overline H}$ is separable. Indeed, if $H = s_\vs(X)$ for $\vs\in\wh X^\N$, then one has the chain of subgroups $\n_\vs(X) \leq H \leq \overline H \leq X$. If $\overline H=X$, then $\wh{X/\overline H}$ is trivially separable. Otherwise $\overline H$ has infinite index in $X$ since $X/\overline H$ is divisible, so $\wh{X/\n_\vs(X)}$ is separable by Corollary \ref{T12}; since $\overline H$ contains $\n_\vs(X)$, the quotient group $X/\overline H$ is a quotient group of $X/\n_\vs(X)$. Therefore, $\wh{X/\overline H}$ is isomorphic to a subgroup of the separable group $\wh{X/\n_\vs(X)}$, so $\wh{X/\overline H}$ is separable as well (see \cite{CI}). Furthermore, $\overline H$ is dually closed (so $X/\overline H$ is MAP) by Fact \ref{lcafact}.

If $H$ is not closed, Corollary \ref{CHK:new} gives that $H \in \CH_K(X)$. If $H$ is a proper closed subgroup of $X$, then $H$ has infinite index as $X/H$ is divisible, so $H \in \CH_K(X)$ by Theorem \ref{T1122}. This proves the inclusion $\CH(X)\setminus \{X\} \subseteq \CH_K(X)\setminus \{X\}$, which along with the obvious inclusion $\CH_K(X)\subseteq \CH(X)$, proves  the equality $\CH(X)\setminus \{X\} = \CH_K(X)\setminus \{X\}$. 

It remains to consider the (closed) subgroup $H=X$ which obviously belongs to $\CH(X)$. If $X$ is compact, then $X\not \in \CH_K(X)$, by Lemma \ref{CKnotOpen}, so $\CH_ K (X) = \CH(X)\setminus \{X\}$. If $X$ is not compact, then  $H = X \in \CH_K(X)$ by Theorem \ref{compactnonac} and Remark \ref{wlog11}(ii). Hence, $ \CH_K(X)=\CH(X)$ in this case.
\end{proof}

In particular, the above corollary yields $\CH_K(\T)=\CH(\T)\setminus \{\T\}$ and $\CH_K(\R)=\CH(\R)$.

\begin{remark}\label{JapJap}
As we shall see in Corollary \ref{JapJap0}, connectedness is necessary in this corollary.
\end{remark}

%

\section{$N$-characterized subgroups}\label{Nsec}

The following consequence of Lemma \ref{PnPAnChar} gives a sufficient condition for a characterized subgroup to be $N$-characterized: 

\begin{cor}
Let $X$ be a topological abelian group and $H$ a subgroup of $X$ which is not autocharacterized.
If $H\in\CH(X)$, then $H\in\CH_N(X)$.
\end{cor}
\begin{proof}
Let $H = \sv(X)$ for some $\vs\in\wh X^\N$.  Then $H\le\n_{\vs_{(m)}}(X)$ for some $m\in\N$ by Lemma \ref{PnPAnChar}. Since $H=\sv(X)=s_{\vs_{(m)}}(X) \ge \n_{\vs_{(m)}}(X)$, we deduce that $H = \n_{\vs_{(m)}}(X)$.
\end{proof}

Here comes an easy criterion establishing when a subgroup is $N$-characterized.

\begin{teo}\label{TsupeGeneral}
Let $X$ be a topological abelian group and $H$ a subgroup of $X$. The following conditions are equivalent:
\begin{itemize}
\item[(i)] $H$ is closed and there exists a continuous injection $X/H\to \T^\N$;
\item[(ii)]  $H\in\CH_N(X)$;
\item[(iii)]  $H$ is closed and $\{0\}$ is $\G$ in $(X/H)^+$.
\end{itemize}
\end{teo}
\begin{proof}
%
(i)$\Rightarrow$(ii) Suppose that there exists a continuous injection $j:X/H\to\T^\N$. Let $\pi:X\to X/H$ be the canonical projection. For every $n\in\N$, let $p_n:\T^\N\to \T$ be $n$-th projection and let $v_n=p_n\circ j\circ \pi$. 
\begin{equation*}
\xymatrix{X/H \ar[r]^{j} & \T^\N \ar[d]_{p_n} \\
X \ar[u]^{\pi} \ar[r]_{v_n} & \T}
\end{equation*}
Therefore, $v_n\in\wh{X}^\N$ for every $n\in\N$ and $H=\n_\vs(X)$, where $\vs=(v_n)$.

(ii)$\Rightarrow$(i) Let $H=\n_\vs(X)$ for $\vs\in\wh X^\N$. Let $\pi:X\to X/\n_\vs(X)$ be the canonical projection and define $j:X/\n_\vs(X)\to \T^\N$ by $j(\pi(x))=(v_n(x))_{n\in\N}$ for every $x\in X$. Since $\n_\vs(X)=\bigcap_{n\in\N}\ker v_n$, then $j$ is well-defined and injective. Moreover, $j$ is continuous.

Finally, (i) and (iii) are obviously equivalent. 
\end{proof}

The above criterion simplifies in the case of open subgroups: 

\begin{cor}\label{Popen}
Let $X$ be a topological abelian group and $H$ an open subgroup of $X$. The following conditions are equivalent:
\begin{itemize}
\item[(i)] $H\in\CH(X)$;
\item[(ii)] $H\in\CH_N(X)$;
\item[(iii)] $[X:H]\le\cc$.
\end{itemize}		
\end{cor}
\begin{proof}
(ii)$\Rightarrow$(i) is obvious and (i)$\Rightarrow$(iii) is Corollary \ref{<c}.

(iii)$\Rightarrow$(ii) Since $H$ is open, $X/H$ is discrete. By hypothesis $|X/H|\leq\cc$, so there exists a continuous injection $X/H\to\T^\N$. Hence,  $H\in\CH_N(X)$ by Theorem \ref{TsupeGeneral}.
\end{proof}


The next is another consequence of Theorem \ref{TsupeGeneral}.

\begin{cor}\label{TitemiPre}
Let $X$ be a metrizable precompact abelian group and $H$ a subgroup of $X$. The following conditions are equivalent:
\begin{itemize}
\item[(i)] $H$ is closed;
\item[(ii)] $H$ is closed and $H\in\CH(X)$;
\item[(iii)] $H\in\CH_N(X)$.
\end{itemize}
\end{cor}
\begin{proof}
(iii)$\Rightarrow$(ii) and (ii)$\Rightarrow$(i) are clear. 

(i)$\Rightarrow$(iii) Since $X$ is metrizable, $X/H$ is metrizable as well, and so $\{0\}$ is $\G$ in $X/H$. By Lemma \ref{LinjPre} there exists a continuous injective homomorphism $X/H\to\T^\N$. Therefore, $H$ is $N$-characterized by Theorem \ref{TsupeGeneral}.
\end{proof}

According to Theorem \ref{T1122}, one can add to the equivalent conditions in Corollaries \ref{Popen} and \ref{TitemiPre} also ``$H\in\CH_K(X)$", in case $[X:H]\geq \omega$. 

\medskip
In the sequel we consider the case of locally compact abelian groups. The following theorem was proved in \cite[Theorem B]{DG13} for compact abelian groups. 

\begin{teo}\label{ThB}
Let $H$ be a locally compact abelian group and $H$ a subgroup of $X$. Then $H\in\CH(X)$ if and only if $H$ contains a closed $\G$-subgroup $K$ of $X$ such that $H/K\in\CH(X/K)$, where $X/K$ is a metrizable locally compact abelian group.
\end{teo}
\begin{proof} 
The equivalence follows from Lemma \ref{TnewTeo}, since by Lemma \ref{Rnv}(iv) the subgroup $\n_\vs(X)$ is closed and $\G$ for every $\vs\in\wh X^\N$. Since $K$ is closed and $\G$, $X/K$ is a metrizable locally compact abelian group. 
\end{proof}

By Lemma \ref{Rnv}, $\n_\vs(X)$ is always closed and characterized.  Theorem \ref{Titemi} describes the closed characterized subgroups of the locally compact abelian groups $X$ by showing that these are precisely the $N$-characterized subgroups of $X$.

\begin{teo}\label{Titemi}
Let $X$ be a locally compact abelian group and $H$ a subgroup of $X$. The following conditions are equivalent:
\begin{itemize}
\item[(i)] $H$ is closed and $H\in\CH(X)$;
\item[(ii)] $H\in\CH_N(X)$;
\item[(iii)] $H$ is closed and $\G$ in the Bohr topology;
\item[(iv)] $H$ is closed, $\G$ and $[X:H]\le\cc$.
\item[(v)]  $H$ is closed and $\wh{X/H}$ is separable.
\end{itemize}
\end{teo}

\begin{proof}
(iii)$\Rightarrow$(ii) follows from Theorem \ref{TsupeGeneral}, (ii)$\Rightarrow$(i) by Lemma \ref{Rnv} and (ii)$\Leftrightarrow$(v) is Corollary \ref{T12}.

(iv)$\Rightarrow$(iii) The group $X/H$ is locally compact, metrizable and has cardinality at most $\cc$, therefore by Theorem \ref{Tinjection} there exists a continuous injective homomorphism $j:X/H\to\T^\N$. Then Theorem \ref{TsupeGeneral} gives the thesis.

(i)$\Rightarrow$(iv) Let $H=\sv(X)$ for $\vs\in\wh X^\N$, and let $\pi:X\to X/\n_\vs(X)$ be the canonical projection. By Corollary \ref{<c}, $[X:H]\le\cc$. By Lemma \ref{Rnv}, $\n_\vs(X)\le\sv(X)$ and $\n_\vs(X)$ is closed and $\G$. Then $X/\n_\vs(X)$ is a metrizable locally compact abelian group, and by hypothesis $\sv(X)$ is closed. Therefore, $\sv(X)/\n_{\vs}(X)$ is closed and hence $\G$ in $X/\n_\vs(X)$. Therefore, $\sv(X)=\pi^{-1}(\sv(X)/\n_{\vs}(X))$ is closed and $\G$ in $X$.
\end{proof}

We see now the following result from \cite{DG13} as a consequence of Theorem \ref{Titemi}. 
	
\begin{cor}[{\cite[Theorem A]{DG13}}]\label{TGchar}
Let $X$ be a compact abelian group and $H$ a closed subgroup of $X$. Then $H\in\CH(X)$ if and only if $H$ is $\G$.
\end{cor}
\begin{proof} 
If $H$ is characterized, then $H$ is $\G$ by Theorem \ref{Titemi}. Viceversa, assume that $H$ is $\G$. Then $X/H$ is compact and metrizable, hence $\la X/H\ra \leq \cc$. So again Theorem \ref{Titemi} implies that $H$ is characterized.
\end{proof}


In the following theorem we use that the $\G$-subgroups of a locally compact abelian group are always closed (see \cite[Theorem A.2.14]{Impi}).
	
\begin{teo}\label{CKNG}
Let $X$ be a compact abelian group and $H$ a subgroup of $X$. The following conditions are equivalent:
\begin{itemize}
 \item[(i)] $H\in\CH_K(X)$ and $H$ is closed; 
 \item[(ii)] $H$ is $\G$ and non-open;
 \item[(iii)] $H\in\CH_N(X)$ and $H$ is non-open;
 \item[(iv)] $H\in\CH(X)$ and $H$ is closed and non-open. 
\end{itemize} 
\end{teo}
\begin{proof}
(i)$\Rightarrow$(iv) Since Lemma \ref{CKnotOpen} implies that $H$ is non-open, (iv)$\Leftrightarrow$(iii) by Theorem \ref{Titemi} and (iv)$\Leftrightarrow$(ii) by Corollary \ref{TGchar}.

(ii)$\Rightarrow$(i)  Since $X/H$ is a metrizable compact non discrete (hence infinite) abelian group, $\{0\}$ is closed and non-open in $X/H$, hence $\{0\}$ is $K$-characterized in $X/H$ by Corollary \ref{CcomKchar}. Therefore, $H$ is $K$-characterized by Proposition \ref{PcntrimgChar}.
\end{proof}

This theorem generalizes Corollary \ref{TGchar} as it implies that,  for a closed non-open subgroup $H$ 
 of a compact abelian group $X$, one has 
$$H\ \text{is}\ \G\ \Leftrightarrow\ H\in\CH(X)\ \Leftrightarrow\ H\in\CH_N(X)\ \Leftrightarrow\ H\in\CH_K(X).$$

The following immediate consequence of Corollaries \ref{Popen} and \ref{Kcar=card} shows that for a discrete abelian group all characterized subgroups are $N$-characterized.

\begin{cor}\label{DiscGr}
Let $X$ be an infinite discrete abelian group and $H$ a subgroup of $X$. The following conditions are equivalent:
\begin{itemize}
\item[(i)] $H\in\CH(X)$;
\item[(ii)] $H\in\CH_N(X)$;
\item[(iii)] $H\in\CH_K(X)$;
\item[(iv)] $[X:H]\le\cc$.
\end{itemize}
\end{cor}

\section{$T$-characterized closed subgroups of compact abelian groups}\label{SSMM}\label{ChTBs}

In \cite[Therorem 4]{Gab10b} Gabriyelyan observed that if $\us$ is a $T$-sequence of an infinite countable abelian group $G$, then  (see \eqref{sigmaeq} for the definition of $\sigma_\vs$)  
\begin{equation}\label{EqT4gab}
\n(G,\sigma_\vs)\cong\sv(\widehat{G_d})^\perp \text{ algebraically},
\end{equation}
where  $G_d$  denotes the abelian group $G$ endowed with the discrete topology. Therefore, the following fact is an immediate consequence of \eqref{EqT4gab}.

\begin{fact}\label{FMAP}
Let $\vs$ be a $T$-sequence of an infinite countable abelian group $G$. Then:
\begin{itemize}
\item[(i)] $(G,\sigma_{\vs})$ is MAP if and only if $\vs$ is a $TB$-sequence; 
\item[(ii)] $(G,\sigma_{\vs})$  is MinAP if and only if $\sv(\widehat{G_d})=\{0\}$ and  $G=\langle\vs\rangle$. 
\end{itemize}
\end{fact}

Recall that a topological abelian group $X$ is \emph{almost maximally almost periodic} (briefly, \emph{AMAP}) if $\n(X)$ is finite.

\begin{rem}
In relation with Fact \ref{FMAP}, Luk\'acs in \cite{Luk06} found a $T$-sequence in $\Z(p^\infty)$ that is not a $TB$-sequence, providing in this way an example of an AMAP group. More precisely, he found a characterizing sequence $\vs$ for $p^m\Jp\le\Jp$ for a fixed $m\in\N_+$, i.e., $\sv(\Jp)=p^m\Jp$. In this way, being $\Jp/p^m\Jp$ finite, then \[\sv(\Jp)^\perp=\n(\Z(p^\infty),\sigma_\vs)\neq\{0\}\text{ is finite.}\] Therefore, $(\Jp,\sigma_\vs)$ is AMAP. Further results in this direction were obtained by Nguyen \cite{Ngu09}. Finally, Gabriyelyan in \cite{Gab09} proved that an abelian group $G$ admits an AMAP group topology if and only if $G$ has non-trivial torsion elements.
\end{rem}

The following theorem, due to Gabriyelyan, links the notions of $T$-characterized subgroup and MinAP topology. 

\begin{teo}\emph{\cite{Gab14}}\label{Tgab14}
Let $X$ be a compact abelian group and $H$ a closed subgroup of $X$. Then $H\in\CH_T(X)$ if and only if $H$ is $\G$ and $H^\perp$ carries a MinAP topology. 
\end{teo}

Following \cite[\S 4]{DPS90}, for a topological abelian group $X$ and a prime number $p$, we denote by $T_p(X)$\index{$T_p(X)$} the closure of the subgroup $X_p=\{x\in X:p^nx\to 0\}$. In case $X$ is compact, one can prove that 
\begin{equation}\label{LAAAAst}
T_p(X) = \{mX: m\in \N_+, (m,p) = 1\}.
\end{equation}
In particular, $T_p(X)$  contains the connected component $c(X)$ of $X$. More precisely, if $X/c(X) = \prod_{p\in\Prm} (X(c(X))_p$ is the topologically primary decomposition of the totally disconnected compact abelian group $X/c(X)$, then 
$$T_p(X)/c(X) \cong  (X/c(X))_p = T_p(X/c(X)).$$ 

Following \cite{DS10}, we say that $d \in \N$ is a proper divisor of $n \in \N$ provided that $d \not\in \{0,n\}$ and $dm =n$ for some $m \in \N$.
Note that, according to our definition, each $d \in\N \setminus \{0\}$ is a proper divisor of $0$.

\begin{defin}\label{def:eo}
Let $G$ be an abelian group.
\begin{itemize}
\item[(i)] For $n \in \N$ the group $G$ is said to be of \emph{exponent} $n$ (denoted by $exp(G)$) if $nG = \{0\}$, but $dG\not = \{0\}$ for every proper divisor $d$ of $n$. We say that $G$ is \emph{bounded} if $exp(G)>0$, and otherwise that $G$ is \emph{unbounded}.
\item[(ii)] \cite{GK03} If $G$ is bounded, the essential order $eo(G)$ of $G$ is the smallest $n\in\N_+$ such that $nG$ is finite. If $G$ is unbounded, we define $eo(G) = 0$.
\end{itemize}
\end{defin}

In the next theorem we aim to give a detailed description of the closed characterized subgroups $H$ of $X$ that are not $T$-characterized.
As stated in Corollary \ref{TGchar}, a closed subgroup $H$ of a compact abelian group $X$ is characterized if and only if $H$ is $\G$ (i.e., $X/H$ is metrizable).  This explains the blanket condition imposed on $H$ to be a $\G$-subgroup of $X$. 

\begin{teo}\label{MainThChapter7}
For a compact abelian group $X$ and a $\G$-subgroup $H$ of $X$, the following conditions are equivalent: 
\begin{itemize}
\item[(i)] $H\not\in\CH_T(X)$;
\item[(ii)] $H^\perp$ does not admit a MinAP group topology; 
\item[(iii)] there exists $m\in\N_$ such that $m(X/H)$ is finite and non-trivial;
\item[(iv)] $eo(X/H) < exp(X/H)$;
\item[(v)] there exists a finite set $P$ of primes so that:
\begin{itemize}
\item[(a)] $T_q(X)\leq H$ for all $q\in\Prm\setminus P$,
\item[(b)] for every $p \in P$ there  exist $k_p\in \N$ with $p^{k_p}T_p(X) \leq H$,
\item[(c)] there exists $p_0 \in P$ such that $p_0^{k_{p_0}-1}T_{p_0}(X) \not\leq H$ and $p_0^{k_{p_0}-1}T_{p_0}(X) \cap H$ has finite index in $p_0^{k_{p_0}}T_{p_0}(X)$;  
\end{itemize}
\item[(vi)] there exists a finite set $P$ of primes so that $X/H \cong \prod_{p\in P} K_p$, where each $K_p$ is a compact $p$-group and there exist some $p_0\in P$ and $k \in \N$ such that $p_0^kK_{p_0}$ is finite and non-trivial. 
\end{itemize}
\end{teo}
\begin{proof}
(i)$\Leftrightarrow$(ii) is Theorem \ref{Tgab14} and (iii)$\Leftrightarrow$(iv) is clear from the definition.
%
%

(ii)$\Leftrightarrow$(iii) The main theorem in \cite{DS14} states that an abelian group $G$ does not admit a MinAP group topology precisely when there exists $m\in\N_+$ such that $mG$ is finite and non-trivial. In our case  $G = H^\perp$ is topologically isomorphic to $\wh{X/H}$, so $G$ does not admit a MinAP group topology if and only if $m\widehat{X/H} = \widehat{m(X/H)}$ is finite and non-trivial, and this occurs precisely when $m(X/H)$ is finite and non-trivial.

%

(vi)$\Rightarrow$(v) Write $X/H \cong \prod_{p\in P} K_p$, where each $K_p$ is a compact $p$-group. Let $p^{k_p}= exp(K_p)$ for every $p\in P$ and let $p_0^kK_{p_0}$ be finite and non-trivial for $p_0\in P$ and $k \in \N$. 
 
Obviously, all $q\in\Prm\setminus P$ are coprime to $m=exp(X/H)$.  As $mX \leq H$, we deduce from \eqref{LAAAAst} that 
\begin{equation}\label{Next}
T_q(X)\leq H,\ \text{for all}\  q\in\Prm\setminus P.
\end{equation} 
This proves (a). From \eqref{Next} we deduce that that $c(X)\leq H$. 

The quotient groups $X' = X/c(X)$ and $H' = H/c(X)$ are totally disconnected, hence  $X' = \prod_{p\in\Prm} X_p'$ and $H' = \prod_{p\in\Prm} H_p'$. Here 
\begin{equation}\label{Last}
X_p' = T_p(X)/c(X)\ \mbox{  and  }\ H_p' = T_p(H)/c(X)\ \text{for every}\  p\in\Prm.
\end{equation} 
Furthermore,  $X' = \prod_{p \in P} X'_p  \times \prod_{q\in\Prm\setminus P} X'_q $. From \eqref{Next} we deduce that $X'_q \leq H_q'$ for all $q\in\Prm\setminus P$. Therefore, $\prod_{q\in\Prm\setminus P} X'_q \leq H'$ and $H' = \prod_{p \in P} H'_p \times \prod_{q\in\Prm\setminus P} X'_q$. Hence, $X/H =  \prod_{p \in P}X_p'/ H'_p $ and consequently, $X_p'/ H'_p  \cong K_p$ for all $p \in P$. Thus, $p^{k_p}X_p' \leq H_p'$ for all $p\in P$. Equivalently, $ p^{k_p}T_p(X) \leq H$ for $p \in P$. This proves (b).
 
As $p_0^{k}K_{p_0}$ is finite and non-trivial, we deduce that $k< k_{p_0}$. Therefore, $p_0^{k_{p_0}-1}K_{p_0}$ is still finite and non-trivial. Hence, $p_0^{k_{p_0}-1}X'_{p_0} \not \leq H_{p_0}'$, and so $$p_0^{k_{p_0}-1}T_{p_0}(X) \not \leq H.$$
To prove the second assertion in  (c) note that the finiteness of $p_0^{k}K_{p_0}$ yields that 
$$p_i^{k_{p_0}-1}(X_{p_0}'/H_{p_0}')= (p_0^{k_{p_0}-1}X_{p_0}' + H_{p_0}')/H_{p_0}' \cong p_0^{k_{p_0}-1}X_{p_0}'/ (H_{p_0}' \cap p_0^{k_{p_0}-1}X_{p_0}' )$$
is finite. 
Hence, from \eqref{Last} we deduce that $T_{p_0}(H)\cap p_0^{k_{p_0}-1}T_{p_0}(X)$ has finite index in $p_0^{k_{p_0}-1}T_{p_0}(X)$.  Therefore, 
$$ H\cap p_0^{k_{p_0}-1}T_{p_0}(X) =\cap T_{p_0}(H)\cap p_0^{k_{p_0}-1}T_{p_0}(X)$$ has finite index in $p_0^{k_{p_0}-1}T_{p_0}(X)$.  

(v)$\Rightarrow$(iii) Let $m'$ be the product of all $p^{k_p}$ when $p$ runs over $P$ and let $m= m'/p$. Then an argument similar to the above argument shows that $m(X/H) \ne \{0\}$ is finite. 
\end{proof}

\begin{cor}\label{AmmazzaGabr}
Let $X$ be a compact abelian group and $H$ a closed subgroup of $X$ that does not contain the connected component $c(X)$ of $X$. Then $H\in\CH_T(X)$ if and only if $H\in\CH(X)$. 
\end{cor}
\begin{proof} 
Clearly, $H$ $T$-characterized implies $H$ characterized. So, assume that $H$ is characterized. Then $H$ is $G_\delta$ in $X$ by Corollary \ref{TGchar}  Let $\pi: X \to X/H$ be the canonical projection. Then $\pi(c(G))$ is a non-trivial connected subgroup of $X/H$, hence $X/H$ is unbounded as its connected component $c(X/H)$ is a non-trivial divisible subgroup.
According to Theorem \ref{MainThChapter7}, $H$ is $T$-characterized. 
\end{proof}

By Corollary \ref{AmmazzaGabr}, for a connected compact abelian group $X$ and $H$ a closed subgroup of $X$,
\begin{equation}\label{eqclosed}
H\in \CH_T(X)\ \Leftrightarrow\ H\in\CH(X)\setminus\{X\}.
\end{equation} 
This result was obtained in \cite{Gab14}; actually, the following more precise form holds (the equivalence (i)$\Leftrightarrow$(ii) is proved in \cite[Theorem 1.14]{Gab14}): 
%

\begin{cor}\label{JapJap0}
For a compact abelian group $X$, the following conditions are equivalent:
\begin{itemize}
\item[(i)] $X$ is connected; 
\item[(ii)] $H\in\CH_T(X)\ \Leftrightarrow\ H\in\CH(X)\setminus\{X\}$ for every closed subgroup $H$ of $G$; 
\item[(iii)] $H\in\CH_K(X)\ \Leftrightarrow\ H\in\CH(X)\setminus\{X\}$ for every closed subgroup $H$ of $G$.
\end{itemize}
\end{cor}
\begin{proof}
(i)$\Rightarrow$(ii) by \eqref{eqclosed} 
and 
(ii)$\Rightarrow$(iii) is obvious.

(iii)$\Rightarrow$(i) Assume that $X$ is not connected. Then $X$ has a proper open subgroup $H$, as the connected component of $X$ is an intersection of clopen subgroups (see \cite{HR79}). Then $H\in\CH(X)\setminus\{X\}$, but $H\not \in \CH_K(X)$ by Lemma \ref{CKnotOpen}.  
\end{proof}

Obviously, each one of the above equivalent conditions implies that $H\in\CH_T(X)\ \Leftrightarrow\ H\in\CH_K(X)$ for every closed
 subgroup $H$ of a compact abelian group $X$. 
To see that in general the latter property is strictly weaker than $H\in\CH_T(X)$, consider the group $X = \Z(3)\times \Z(2)^\N$. Then for the closed subgroup $H = \{0\}$ of $X$ one has $H\in \CH_K(X)$ by Theorem \ref{CKNG}, while 
$H\not\in \CH_T(X)$ by Theorem \ref{MainThChapter7}, as $3X$ is finite and non-trivial (moreover, $\wh X$ does not admit any MinAP group topology, as noticed by D. Remus, see \cite{Com}).

\section{Final comments and open questions}\label{finalsec}

 In this section we collect various open questions arising throughout the paper.

\medskip
For a topological abelian group $X$ and $\vs\in\wh X^\N$, we defined for each $m\in\N$ the closed subgroup $\n_{\vs_{(m)}}(X)$ in \eqref{vmeq}. Since these subgroups are contained one in each other, the increasing union $\mathbf F_\vs(X):=\bigcup_{m\in\N}\n_{\vs_{(m)}}(X)$ is an $F_\sigma$-subgroup of $X$ and it is contained in $s_\vs(X)$. We do not know whether $\mathbf F_\vs(X)$ is characterized or not:

\begin{que} 
For a topological abelian group $X$ and $\vs\in\wh X^\N$, is $\mathbf F_\vs(X)$ a characterized subgroup of $X$?
\end{que}

This question is motivated by \cite[Theorem 1.11]{DG13}, where it is proved that under some additional restraint, the union of a countably infinite increasing chain of closed characterized subgroups of a metrizable compact abelian group is still characterized. 
On the other hand, it is known that every characterized subgroup is $F_{\sigma\delta}$ (see Lemma \ref{Lbasic}(iv)) and that characterized subgroup need not be $F_\sigma$.

\medskip
In analogy to $T$-characterized subgroups, we have introduced here the notion of $TB$-characterized subgroup (see Definition \ref{TB}). In analogy to what is already known for $T$-characterized subgroups, one could consider the following general problem.

\begin{pro}\label{TBPB}
Study the $TB$-characterized subgroups of topological abelian groups.
\end{pro}

Next comes a more precise question on the properties of $T$- and $TB$-characterized subgroups. In fact, we do not know whether the counterpart of Corollary \ref{NEW:lemmaK} is true for $T$- and $TB$-characterized subgroups:

\begin{que}
Let $X$ be a topological abelian group and $X_0$, $X_1$, $X_2$ subgroups of $X$ with $X_0\leq X_1\leq X_2$ such that $X_1$ is dually embedded in $X_2$. 
\begin{itemize}
\item[(i)] If $X_0\in \CH_T(X_1)$ and $X_1\in \CH_T(X_2)$, is then $X_0\in \CH_T(X_2)$?
\item[(ii)] If $X_0\in \CH_{TB}(X_1)$ and $X_1\in \CH_{TB}(X_2)$, is then $X_0\in \CH_{TB}(X_2)$?
\end{itemize}
\end{que}

A full description of open $K$-characterized subgroups is given in Theorems \ref{Kcar=car}  and \ref{PtimesKchar:new}, while Theorem \ref{T1122} describes the closed $K$-characterized subgroups of infinite index that are also $N$-characterized. Moreover, $N$-characterized closed subgroups of infinite index are $K$-characterized. This leaves open the following general problem and question.


\begin{pro}
For a topological abelian group $X$, describe $\CH_K(X)$.
\end{pro}

\begin{que}
Let $X$ be a topological abelian group.
Can one add the ``$H\in\CH_K(X)$" as an equivalent condition in Theorem \ref{T1122}?
Equivalently, does there exist a closed subgroup $H$ of $X$ such that $H\in\CH_K(X)\setminus \CH_N(X)$?
\end{que}

In Theorem \ref{PtimesKchar:new} we have seen in particular that a proper open finite-index subgroup $H$ of a topological abelian group $X$ is autocharacterized precisely when $H\in \CH_K(X)$. We do not know whether also the stronger condition $H\in\CH_T(X)$ is equivalent:

\begin{que}
Let $H$ be a topological abelian group and $H$ an open subgroup of $X$ of finite index. Does $H\in\CH_T(X)$ whenever $H$ is autocharacterized? What about the case when $H$ is a topological direct summand of $X$? 
\end{que}

By looking at Theorem \ref{TsupeGeneral} and Corollary \ref{TitemiPre}, the following natural question arises:

\begin{que}\label{Wis}
Are closed $\G$-subgroups of a precompact abelian groups always $N$-characterized? 
\end{que}

This amounts to ask whether there exists a continuous injection from $X/F$ into $\T^\N$ for every closed $\G$ subgroup $F$ of a precompact abelian group $X$, in other words we are asking for a generalization of Lemma \ref{LinjPre}.

\end{document}